\documentclass[12pt]{article}
\usepackage{setspace}
\usepackage{multirow}
\usepackage{amssymb}
\usepackage{amsmath}
\usepackage{amsfonts}
\usepackage{graphicx}
\usepackage{extarrows}
\usepackage{fancyhdr}
\usepackage{verbatim}
\usepackage{indentfirst}
\usepackage{booktabs}
\usepackage[subnum]{cases}
\usepackage{bm}
\usepackage{amsthm}
\usepackage{subfigure}
\usepackage{float}
\usepackage{algorithm}
\usepackage{algorithmic}
\usepackage{cite}
\usepackage{array}
\newdimen\abovekeywordsskip
\newdimen\abovetheoremskip

\newtheorem{thm}{Theorem}

\newtheorem{example}{Example}
\newtheorem{remark}{Remark}

\newtheorem{df}{Definition}

\newcommand{\bqn}{\begin{equation}}
\newcommand{\eqn}{\end{equation}}

\usepackage{color}

\begin{document}
%\begin{CJK*}{GBK}{song}
%\begin{spacing}{1.0}
%\title{\Large \bf Superiorized iteration and its application to XCT image reconstruction }
%\title{\Large Superiorization of iteration algorithm based on proximate point method and its application to XCT image reconstruction}
\title{\large Superiorized iteration based on proximal point method and its application to XCT image reconstruction
\footnotetext{This research is supported by  NSFC (11471101, 1140117), Foundation of Education Department of Henan Province, China (14B110019),  Natural Science Foundation of  Henan Province, China (132300410150).}}
%\subjclass[2000]{92C55, 15A29}
%\thanks{Project supported by the National Natural Science Foundation of China
%(No.11001074, No.11061015, No.11101124) }
%\author{Bill$^1$\\University, Shanghai , China}
\author{\small Shousheng Luo$^1$, \small Yanchun Zhang$^1$,  \small Tie Zhou$^2$, \small Jinping Song$^1$ \\
\small 1. School of Mathematics and Statistics,  Henan University, Kaifeng, 475004, China,\\
\small 2. School of Mathematical Sciences, Peking University, Beijing, 100871, China}
\date{ }
\maketitle
\begin{minipage}{13.2cm}
\footnotesize {\bf Abstract}
In this paper, we investigate  how to determine a  better perturbation for superiorized  iteration.
We propose to seek the perturbation
by proximal point method. In our method, the direction and amount of perturbation are
computed simultaneously.
The convergence conditions are also discussed for bounded perterbation resilence iteration.
Numerical experiments on simulated XCT projection data show that the proposed method
improves the convergence rate and the image quality.
\end{minipage}

\vskip8pt
\begin{minipage}{13.2cm}
\end{minipage}
\vskip 10pt
%\footnotesize {\bf 2000 MR Subject Classification}\quad  32A30,30C45
%\vskip24pt
%%%%%%%%%%%%%%%%%%%%%%%%%%%%%%%%%%%%%%%%%%%%%%%%%%%%%%%%
\section{Introduction}
Linear imaging problems, such as X-ray computed tomography(XCT), single-photon emission computed tomography(SPECT) and magnetic resonance imaging (MRI) can be formulated by
\begin{eqnarray}
Ax=b,\label{eq_sys}
\end{eqnarray}
where the imaging matrix $A\in \mathbb{K}^{m\times n}$,
the observed data $b\in\mathbb{K}^m$, and $x\in \mathbb{K}^n$ is the image to be reconstructed\cite{Herman_book,Natterer2001}. $\mathbb{K}$ can be the real number field $\mathbb{R}$ or the complex number field $\mathbb{C}$.
%Our aim is to reconstruct the original $x$ from the observed data $b$ with image system $A$ given.

Iterative methods, such as algebraic reconstruction technique(ART)\cite{Herman1970,CensorBIP,CensorSAP} and expectation maximum(EM)\cite{SheppEM},
are usually used  to solve (\ref{eq_sys}) because of
the ill-posedness of $A$ and huge data dimension for practical problems\cite{Herman_book,JiangDeve}.
However, the reconstructed results by iteration methods mentioned above are
not satisfied when the $m\ll n$, such as few-angle and sparse-angle XCT.
Regularization methods based on optimization
are investigated to improve the image quality (see \cite{Pan2006,PanPD,HermanSmall} and references therein). However, we are short of efficient algorithm to solve the
optimization problem because the dimension is huge for imaging problems\cite{Sup_BIP,Censor-Sup}.
Superiorization of iteration was proposed to
seek a desirable result from the application point of view at a relatively low computational cost.

In order to introduce the concept of superiorization, we consider the following
 constrained optimization  problem for (\ref{eq_sys})
\begin{eqnarray}
\min \phi(x)\quad \text{subject to }\quad Ax=b \text{ and } x\in C_0\label{const_eq1},%=\bigcap\limits_{i=1}^N C_i,\label{const_eq1}
\end{eqnarray}
where the convex function $\phi$ and the convex set $C_0$ denote the prior knowledge of
the solution $x$. In this paper, we assume that $C_0$ is a bounded set in $\mathbb{K}^n$ because
the desired images are bounded in practice.
Let $C_i=\{x\in\mathbb{R}^n|\langle a^i, x\rangle=b_i\}(i=1,2\cdots,m)$ with $a^i$ denoting the $i$-th row of $A$, and the optimization problem (\ref{const_eq1}) can be written as
\begin{eqnarray}
\min\phi(x)\quad\quad \mathrm{subject~to} \quad\quad x\in C\triangleq\bigcap_{i=0}^mC_i\label{const_eq2}.
\end{eqnarray}
As mentioned above, we can use ART or EM iteration to find a feasible point in $C$\cite{Herman1970,CensorBIP,CensorSAP,JiangDeve}, but we are short of efficient algorithm to solve  (\ref{const_eq2}), i.e. look for a point $x^\ast\in C$ such that $\phi(x^\ast)\leq \phi(x) \forall x\in C$,
since the dimension is huge for imaging problems \cite{Sup_BIP,Censor-Sup}.
The aim of superiorized iteration is to steer the iterates toward  a superior point $x^\star\in C$ such that $\phi(x^\star)$ is smaller but not necessarily smallest.
%In fact, we can view the superiorization algorithm as a technique lies between the feasible problem and the optimimization problem.
%emerging iteration
%technique to steer the iterates toward a point that is feasible and superior, but not necessarily optimal, with respect to the given objective function $\phi$.
% handle the constrained optimization problem
%(\ref{const_eq2}).
%Assume $x^\star$ be the limit of iteration (\ref{piter}),
%the superior solution means that $\phi(x^\star)\leq \phi(x^\ast)$ with $x^\ast$ being the limit of $x^{k+1}=Px^k$ without perturbations.
There are three reasons to propose and use the superiorization approach.
\begin{enumerate}
\item The computational cost to solve the large scale constrained optimization problems  is very high, especially when the constraint set $C$ is a complicated convex set(for example $C$
is the intersection of a series of convex sets).
\item  We need not to solve the constraint  optimization problem (\ref{const_eq2}) precisely since the optimal solution of (\ref{const_eq1}) is possibly not the solution we want in
practical problems\cite{Censor-Sup,CensorBIP}.
For example, the optimal solution often suffers from
staircase problem if $\phi$ is the total variation function\cite{ROF1992,Tai,Chan2005}.
\item The convergence rate could be improved by using appropriate perturbed directions and amounts in our opinion.
\end{enumerate}
%\cite{Vogel2002,optimal1,optimal3,Zeng2009}
%\begin{eqnarray}
%\min \phi(x)\quad \text{s.t. }\quad x\in C=\bigcap\limits_{i=1}^N C_i,\label{const_eq1}
%\end{eqnarray}
%where $\phi$ is a convex function represent prior knowledge of the solution, and $C$ is a convex set called feasible set. In this paper, we assume $C$ is bounded set of $\mathbb{R}^n$, because
%the concerned solutions are bounded in practice.
%Before giving the superiorization iteration pr, we first

Superiorization of iterative methods have been used in various image reconstruction problems, such as XCT \cite{Sup_BIP,Sup_SAP,HermanSmall,Herman-Acc}, SPECT \cite{Luo2014}, bioluminescence tomography \cite{JiangBLT} and proton computed tomography\cite{PCT,PCT2}, since it was proposed. Under the assumption that $x^{k+1}=Px^k$ converge to
a feasible point $x^\ast\in C$,
 the superiorized version of $P$ with respect the objective function $\phi$  (\ref{const_eq1})
 can be written as\cite{Censor-Sup}
\begin{equation}
\left\{
\begin{array}{lcl}
y^k&=&x^k+\beta_kv^k\\
x^{k+1}&=&Py^k
\end{array}\right.,\label{piter}
\end{equation}
where $v^k$ is a descent direction of $\phi$ at $x^k$,
and $\beta_k>0$  such that $\phi(y^k)\leq \phi(x^k)$ and $\mathrm{dist}(x^{k+1},C)<\mathrm{dist}(x^k,C)$\cite{Censor-Sup,Herman-Acc}.
There are two key problems of superiorized iteration (\ref{piter})
to be answered
\begin{enumerate}
\item Under what conditions the iteration sequence of (\ref{piter}) is convergent.
\item Is the limit  $x^\star $in the constraint set $C$ if the sequence is convergent?
\item How to determine better perturbation direction $v^k$ and amount $\beta_k$ to guarantee the convergence of iteration (\ref{piter})?
\item Can we attain our aim that $\phi(x^\star)\leq \phi(x^\ast)$ where $x^\star,x^\ast$ are
the limits of the superiorized iteration sequence and the original iteration sequence without perturbation, respectively.
\end{enumerate}

As far as we know, the studies on superiorized algorithms are about
the first two problems\cite{HermanPert,Herman-Acc,CensorDyn,JiangBPR}.
A condition called bounded perturbation resilience(BPR, see section 2 for the definition) of iteration plays an important role in the proof of perturbed version of iteration algorithm.
For a BPR iteration $P$, the sequence generated by (\ref{piter}) converges to
a  feasible point if the perturbation is summable, i.e.
 $\sum_{k=1}^\infty\|\beta_kv^k\|=\sum_{k=1}^\infty |\beta_k|<\infty$.
 Therefore, it is only required that the perturbation is summable for the BPR iteration $P$
 to guarantee the convergence of iteration (\ref{piter}).
The ART iteration and its variations, such block iteration projection(BIP)\cite{Sup_BIP} and
string averaging projection (SAP)\cite{Sup_SAP}, were proved to be bounded perturbation resilient when the constraints are consistent.
%the perturbed (\ref{piter}) version for various iteration $P$, such as
%algebraic reconstruction technique(ART) \cite{Herman2012,Censor_2010,Dav2009,bibitem18}
%and expectation maximization(EM) iterations\cite{Luo2014,Jiang2015}.
%The convergence of superiorizations of ART iterations,
%block-iterative projection algorithm for example, were proved for the
%consistent constraints that the set $C$  is a nonempty set \cite{bibitem9,bibitem11}.
For inconsistent cases, the authors of \cite{Herman-Acc} proved the convergence of the
symmetric version of ART under summable perturbations. The convergence and perturbation resilience of dynamic
string-averaging projection algorithms were studied in \cite{CensorDyn} recently.
The superiorization of EM algorithm was proposed in \cite{JiangBLT}, and then a modified
version  and some assumptions for convergence were studied in \cite{Luo2014}.
Furthermore, the bounded perturbation resilience(BPR) of EM algorithms
 was proved in \cite{JiangBPR}.

We attempt to study the third problem in this paper. It is obvious that
the variables $\{v^k\}$ and $\{\beta_k\}$ play an important
role in he convergence rate and quality of reconstructed images.
In the literature, parameters $v^k$ and $\beta_k$ are determined by  two successive steps for each iteration. Firstly,
an unit descent direction of $\phi$ at $x^k$  is usually as $v^k$ \cite{Censor-Sup,Herman-Acc,Sup_BIP,Sup_SAP}.
Secondly, the parameters $\beta_k$ is adjusted by the criterion that $\phi(x^k+\beta_kv^k)\leq\phi(x^k)$
and $\text{dist}(P(x^k+\beta_kv^k),C)<\text{dist}(x^k,C)$.
Numerical experiments showed that
one should adjust $\beta_k$ by trial and error such that $\phi(x^k+\beta_kv^k)\leq\phi(x^k)$
and $\text{dist}(P(x^k+\beta_kv^k),C)<\text{dist}(x^k,C)$.
%to satisfy the conditions mentioned above.
And these operations make the computational cost much more high.
%On the other hand, the subgraident is a local property,
% it may be optimal in the global sense.

In this paper, we propose a new method to determine
the perturbed results $y^k$ directly, rather than the middle variables $v^k$ and $\beta_k$.
And the direction and amount of perturbation can be obtained by $\beta_kv^k=y^k-x^k$ simultaneously.  The proposed method determines the perturbed result $y^k$
 by optimization method with single regularization parameter.
 In fact, $y^k$ is a proximate point of $x^k$ with respect  to $\phi$.
We can prove that the inequality  $\phi(y^k)\leq \phi(x^k)$ holds naturally.
Therefore, we only need to adjust the regularization parameter in the optimization problem such that $\text{dist}(P(y^k),C)<\text{dist}(x^k,C)$. Therefore, the computation cost could be
reduced intuitively.
Moreover, numerical experiments on XCT image reconstruction
show that the convergence rate and quality of reconstruction
images can be improved by the proposed method. Furthermore, the convergence
of the proposed algorithm is investigated theoretically.
We call the proposed algorithm as $\phi$-proximate point superiorization ($\phi$-PP superiorization) algorithm in the following
to emphasize the perturbed point $y^k$ is
the $\phi$-proximate point of $x^k$.
As for the fourth problem, there is no progress as far as we known.

%In order to improve the performance of the superiorization algorithm, we proposed to use
%\begin{eqnarray}
%  y^k=\text{arg}\min_y\phi(y)+{1\over 2\lambda_k}\|y-x^k\|_2^2\label{noncon_pert}
%\end{eqnarray}
%as the perturbed result $x^k+v^k$, i.e. the perturbation $v^k=y^k-x^k$. Here $\lambda_k>0$ is user-defined parameters. %, and call $v^k$ as $\phi$-optimal perturbation.
%It is obvious that $\phi(y^k)<\phi(x^k)$. Therefore, we only need to verify the distance reducing condition
%{\red$dist(P(x^k+v^k),C)<dist(x^k,C)$}. Because the perturbation $v^k$ is computed by the optimization problem (\ref{noncon_pert}), we call $v^k$ as $\phi$-optimal perturbation,
%and the corresponding superiorization algorithm with $\phi$-optimal perturbed as $\phi$-optimal superiorization algorithm in the following.
% Indeed, $y^k$ is the $\phi$-proximate point of $x^k$ \cite{chambolle2011,Xu2011}.
%In this paper, we will prove the convergence of the $\phi$-optimal superiorization algorithm
% under the assumption that $\sum_k\lambda_k<\infty$ for bounded perturbation resilient iteration\cite{Censor_2010}.
%
%We validated the $\phi$-optimal superiorization algorithm by CT image reconstruction simulations, although it can be applied to diverse problems.
%The numerical experiments showed the superiorization algorithm can be accelerated dramatically by using the proposed perturbation, and the performance can be improved significantly.

The rest of this paper is organized as follows. In section 2, we present the proposed
superiorization algorithm and theory analysis.
Several numerical experiments are present to demonstrate the efficiency
of the proposed superiorization algorithm in section 3.
Section 4 contains the conclusions and future works.

\section{$\phi$-proximate point($\phi$-PP) superiorization algorithm  and theory analysis}
%This section consists two parts. One is the
\subsection{$\phi$-PP superiorization algorithm}
Let $P$ be an iteration for the feasible problem $x\in \bigcap_{i=0}^m C_i$. The
superiorized version of $P$ can be illustrated algorithm \ref{alg}.
\begin{algorithm}[htb]
\caption{$\phi$-proximate point superiorzation algorithm}
\label{alg}
\begin{algorithmic}[1]
\STATE set $k=0$,~$x^0\in{\mathbb{R}^n}$,~$\beta_0>0$,~and~$0<\gamma<1$.
%\STATR while $k\leq{M}$\\
\STATE while $k\leq{M}$ \text{and} $\textrm{Res}(x^k)\geq{\epsilon}$ \\
\STATE ~~logic=true;
\STATE ~~while(logic)
\STATE ~~~~~~$y^k=\mathrm{perturb}_\phi(x^k,\beta_k)$
%$\arg\displaystyle\mathop{\min}_{y}\{\phi(y)+\frac{1}{2\lambda_{k}}\|y-x^k\|^2_2\}.$
\STATE ~~~~~~$x^{k+1}=Py^k$
\STATE ~~~~~~if  $\phi(y^k)\leq\phi(x^k)$ and $(\textrm{dist}(x^{k+1},C)<\textrm{dist}(x^k,C))$\hfill ($\ast$)
\STATE ~~~~~~~~ logic=false
\STATE ~~~~~~else
 \STATE~~~~~~~~$\beta_{k}={\beta_{k}}{\gamma}$\\
\STATE ~~~~~~end(if)
\STATE~~ end(while)
\STATE ~~ $\beta_{k+1}={\beta_{k}}{\gamma}$
\STATE ~~ $k=k+1$
\STATE end(while)
\end{algorithmic}
\end{algorithm}

There are two methods to compute $\textrm{dist}(x,C)$\cite{Herman-Acc,Sup_BIP},
\begin{eqnarray}
\textrm{dist}(x,C)=\sqrt{\sum_{j=1}^m
\left(\frac{b_j-\langle a^j,x\rangle}{\|a_j\|}\right)^2},
\end{eqnarray}
and
\begin{eqnarray}
\textrm{dist}(x,C)=\sqrt{\sum_{j=1}^m\left(b_j-\langle a^j,x\rangle\right)^2}\triangleq \text{Res}(x).
\end{eqnarray}
The first method  is unstable because small changes in the
data  $x$ result in large changes in the value of dist$(x,C)$ \cite{Herman-Acc}. Therefore,
 we use the second formula to measure the deviation between the projection of
 reconstructed image $x$ and the observed data.

In the classical superiorization algorithm( see \cite{Sup_BIP,Sup_SAP,Herman-Acc} and references therein),
$y^k=\textrm{perturb}_\phi(x^k,\beta_k)=x^k+\beta_kv^k$, where
$v^k=\left\{\begin{array}{cl}
-{u^k\over \|u^k\|},& u^k\neq0\\
0& u^k=0
\end{array}\right.$
with $u^k\in\partial\phi(x^k)$ being fixed as a subgradient of $\phi$ at $x^k$.
Then $\beta_k$ is
adjusted to make the condition ($\ast$) true. Because the perturbation direction $v^k$
is fixed, it is possible to adjust $\beta_k$ many times, and the computational cost is increased.

In this paper, we propose to compute the perturbed vector $y^k$ directly by  solving the following optimization problem
\begin{eqnarray}
y^k&=\text{perturb}_\phi(x^k,\beta_k)=\text{arg}\min_y \phi(y)+{1\over2\beta_k}\|y-x^k\|_2^2\label{eq_iter1}.
\end{eqnarray}
%where $P$ is the feasible operator and the norm $||\cdot||_2$ is defined as $\sqrt{\langle\cdot,\cdot\rangle}$ which is the Euclidean inner probuct $\langle{x},{y}\rangle=\displaystyle\mathop{\sum}_{k=1}^{n}x_k y_k$, for $x,y\in{\mathbb{R}^{n}}$.
Let $\psi(y)=\phi(y)+{1\over2\beta_k}\|y-x^k\|_2^2$, and we have
\begin{eqnarray}
\phi(y^k)\leq \psi(y^k)\leq \psi(x^k)=\phi(x^k),
\end{eqnarray}
Therefore, the first condition in ($\ast$) is naturally true for the proposed method.
Moreover, it is obvious that the perturbation $y^k-x^k$  by our method is dependent on the parameter $\beta_k$, and the perturbation direction and amount change simultaneously along with the change of $\beta_k$, while the classic method change the perturbation amount only.
%Therefore, we only need to {\red validate the second condition in implementation for our method.}

It seems that the optimization problem (\ref{eq_iter1}) could increase the
computational cost of superiorization algorithm. However, the experiments show that the convergence rate is accelerated by using the proposed perturbation, and the computational cost can be reduced by
the same terminated criterion since the number of iteration is smaller. Furthermore, for the commonly used regularization function $\phi$, we have explicit solution or efficient algorithm for the regularization problem (\ref{eq_iter1}), such as
\begin{example}
 $\phi(x)=\|x\|_0$, $y^k_i=\mathrm{hard}_\phi(x^k_i,{\beta_k})=\left\{\begin{array}{cl}x_i&|x_i|>\beta
 \\0&\textrm{otherwise}\end{array}\right.$;
\end{example}
\begin{example}
  $\phi(x)=\|x\|_1$, $y^k_i=\mathrm{soft}_\phi(x^k_i,{\beta_k})=\max\{|x_i|-\beta_k,0\}\textrm{sign}(x_i)$;
  \end{example}
\begin{example}
  $\phi(x)={1\over2}\|x\|_2^2$, $y^k_i={x_i^k\over \beta_k+1}.$
\end{example}
\begin{example}
For the total variation regularization,  $\phi(x)=\sum_{i=1}^n|\nabla x_i|$\cite{ROF1992},
we can apply fast algorithms, such as dual algorithm  \cite{chambolle2004}, splitting Bregman iteration \cite{Bregman},
fixed point method\cite{XuProximity} and ADMM method\cite{yin2008}, to solve  (\ref{eq_iter1}).
In this paper, we the dual algorithm \ref{alg2} (see \cite{chambolle2004} for details) to solve the
subproblem (\ref{eq_iter1}).

\begin{algorithm}[htp]
\caption{Dual algorithm for TV-minization \cite{chambolle2004}.}
\label{alg2}
\begin{algorithmic}
\STATE~Let $p^{0,k}=0$, $0<\tau<1/8$ and $N>0$,\\
\STATE \quad for s=0:(N-1),
%\STATE \qquad
\begin{eqnarray}p_{i,j}^{s+1,k}=
\frac{p_{i,j}^{s,k}+\tau(\nabla(\textrm{div}p^{s,k}-x^{s,k}/\beta_{k}))_{i,j}}
{1+\tau|(\nabla(\textrm{div}p^{s,k}-x^{s,k}/\beta_k))_{i,j}|},
\end{eqnarray}
\STATE \quad end(for)
\STATE $y^k=x^{k}-\beta_{k}\textrm{div}(p^{N,k}) $;
%\STATE \qquad set $s=s+1$;
\end{algorithmic}
\end{algorithm}
\end{example}
\subsection{Convergence of the $\phi$-PP superiorization algorithm}
%%%
We will investigate the convergence of the  $\phi$-PP superiorization
algorithm in this subsection.
Before illustrating the convergence of algorithm \ref{alg}, we first review the BPR condition.
\begin{df}\label{df1}
Bounded perturbation resilience \cite{Censor-Sup}: An iterative operator $P:\mathbb{R}^n\longrightarrow{\mathbb{R}^n}$ is said to be bounded perturbation resilience with respect to a given nonempty set $C\subseteq\mathbb{R}^n$ if the following is true: If a sequence $\{x^k\}_{k=0}^{\infty}$, obtained by the iterative process
\begin{equation}
x^{k+1}=Px^k,~~~~~\textrm{for all}~~~ k\geq0,
\label{noperturbation}
\end{equation}
converges to a point in $C$ for all $x^0\in{\mathbb{R}^n}$, then the iterative sequence $\{x^k\}_{k=0}^{\infty}$
\begin{equation}
x^{k+1} = P(x^k+v^{k}),~\textrm{for all}~ k\geq{0},
\label{eq1}
\end{equation}
also converges to a point in $C$ for all $x^0\in{\mathbb{R}^n}$ provided that  $\displaystyle\mathop{\sum}_{k=1}^{\infty}\|v^{k}\|_2<\infty$.
\end{df}
%%%%%%%
Let $P_i$ be the projection operator on $C_i(i=0,1,\cdots,m)$, i.e.
$
P_ix=\min_{y\in C_i}\|x-y\|_2$.
For $i=1,2\cdots,m$, we have
\begin{eqnarray}
P_i x&=&x+\frac{b_i-\langle{a^i,x}\rangle}{||a^i||^2_2}(a^i)^{t},\qquad i=1,2,\cdots,m.
\end{eqnarray}
In practice we can select simple convex set $C_0$, $C_0=\{x\in\mathbb{R}^n|0\leq x_i\leq L,i=1,2,\cdots,n\}$ for instance, such that $P_0x$ have explicit solution.
The classic ART iteration and its variations (block iteration projection\cite{CensorBIP} and successive
averaging projection\cite{CensorSAP} for example) by the $P_i(i=0,1,\cdots,m)$ are all nonexpansive and BPR  operator \cite{Sup_BIP,Sup_SAP}.
Becuase the desirable solutions of (\ref{const_eq1}) are bounded in practice,
we define $P=P_0\hat{P}$ with  $\hat{P}$ denoting the ART iteration and its variations.
Therefore, the iteration sequence $\{x^k\}\subset C_0$ is bounded, i.e. there is a number $R$ such that $\|x^k\|\leq R$ for all $k\in \mathbb{N}$.

Under the assumption that $P$ is BPR iteration(ART iterations and EM iterations for example),
we only need to prove $\sum_k\|y^k-x^k\|_2<\infty$ to show the convergence of algorithm \ref{alg}, and we have the following theorem about algorithm \ref{alg}.

%In order to prove $\sum_k\|y^k-x^k\|<\infty$, we assume
%that there is $M>0$ such that
%\begin{eqnarray}
%\|Px\|\leq M,\qquad \forall~x\in \mathbb{R}^n.\label{eq_bound}
%\end{eqnarray}
%In fact the solution (which we are interested in) of the problem (\ref{const_eq2}) is bounded.
\begin{thm}\label{thm1}
Assume that $\phi(x)$ is a nonnegative and closed convex function on $\mathbb{R}^n$,
$P$ is a continuous and BPR iteration for feasible problem $x\in\bigcap_{i=0}^nC_i$, and
$Px \in C_0$  is bounded for all $x\in \mathbb{R}^n$.
Then,  we have
\begin{enumerate}
\item The sequence $\{y^k\}$ in algorithm \ref{alg} is bounded,
\item For a given bounded set $B_0\subset \mathbb{R}^n$, there is a real number $M>0$ such that $|\partial\phi(x)|<M$ for all $x\in B_0$,
\item $\sum_{k=0}^{+\infty}\|y^k-x^k\|<+\infty$.
\end{enumerate}
In other word,  the iteration sequence $\{x_k\}$ generated by algorithm \ref{alg} converges to a feasible point in $C$ if $P$ is a BPR operator.
\end{thm}
\begin{proof}
{\bf Statement 1}\\
Firstly, we have the sequence $\{x^k\}\subset C_0$ is bounded based on the assumption of $P$,
and there is a positive number $R$ such that $\|x^k\|\leq R$ for all $k\geq0$.
Secondly, $\phi$ is a continuous function on $C_0$ \cite{convex2004},
and then there are a real number $M_1$ such that $\phi(x)\leq M_1$ for all $x\in C_0$.

Since $y^k$ is the solution of (\ref{eq_iter1}), by selecting $y=x^k$ in (\ref{eq_iter1}) we have
\begin{equation}
\phi(y^k)+\frac{1}{2\beta_{k}}\|y^k-x^k\|^{2}_{2}\leq \phi(x^k).
\end{equation}
Therefore, we have
\begin{equation}
\frac{1}{2\beta_{k}}\left(\|y^k\|^{2}_{2}-\|x^k\|^{2}_{2}\right)
\leq \phi(y^k)+\frac{1}{2\beta_{k}}\|y^k-x^k\|^{2}_{2}\leq \phi(x^k)\leq M_1.\label{ieq2}
\end{equation}
%The sequence $x^k$ is bounded, $\phi$ is a convex function, so $\phi(x^k)$ is bounded. Denoting by $M_2$ a finite upper bound of $\phi(x^k)$,
The equality (\ref{ieq2}) can be written as
\begin{equation}
\|y^k\|\leq \|x^k\|+2\beta_{k}\phi(x^k)\leq R + 2\beta_{k}M_1.
\end{equation}
Because $\{\beta_k\}$ is a positive and summable sequence, there exists a positive number $M_2=R + 2M_1\sum_{k=1}^\infty\beta_{k}$ such that
\begin{equation}
\|y^k\|\leq M_2.
\end{equation}
{\bf Statement 2 }\\
For the given  bounded set $B_0$, let  $B_1=\{x+y|x\in C_0,  y\in\mathbb{R}^n~\text{and}~\|y\|_2\leq1\}$.
Therefore, $B_0\subset B_1$,
and there is a real number $M_3>0$  such that $|\phi(y)|\leq M_3$ for all $y\in B_1$
 by the assumption of $\phi$\cite{convex2004}.
In order to prove the conclusion, we need to prove there exist a real number $M_4$  such that $\max_{\hat{x}\in \partial \phi(x)}\|\hat{x}\|\leq M_4$  for $\forall x\in B_0$.

Let $x\in B_0$ and $\hat{x}\in\partial\phi(x)$. Then for all $z\in\mathbb{R}^n$ we have
\begin{eqnarray}
\phi(z)-\phi(x)\geq\langle \hat{x},z-x\rangle.
\end{eqnarray}
Furthermore, we can choose an appropriate $y \in B_1$ such that
$\langle \hat{x},z-x\rangle\geq\| \hat{x}\|$. In fact, if $\hat{x}=0$, it is obvious that $\langle\hat{x},z-x\rangle\geq\|\hat{x}\|$.
Otherwise, one can choose $z=x+{\hat{x}\over\|\hat{x}\|}\in B_1$ such that $\langle\hat{x},z-x\rangle\geq\|\hat{x}\|$.
Therefore, we have
\begin{eqnarray}
\|\hat{x}\|\leq {|\phi(z)|+|\phi(x)|}\leq 2M_3,
\end{eqnarray}
i.e. $\max_{\hat{x}\in\partial\phi(x)}\|\hat{x}\|\leq 2M_3$ for $\forall~x\in B_0$.
The second statement is true by selection $M_4=2M_3$.
\\
{\bf Statement 3}

Let $y^k$ is the solution of (\ref{eq_iter1}), and then we have
\begin{eqnarray}
0\in\partial \phi(y^k)+{1\over\beta_k}(y^k-x^k).
\end{eqnarray}
Therefore, the $\phi$-PP perturbation $v^k=y^k-x^k\in\beta_k\partial\phi(y^k)$.
Assuming  $B_0=\{x\in\mathbb{R}^n|~\|x\|_2\leq M_2\}$ in the second statement,  we have
$y^k\in B_0$ and
\begin{eqnarray}
\sum_k\|v^k\|=\sum_k\beta_k\|\partial\phi(y^k)\|\leq M_4\sum_k\beta_k<\infty.
\end{eqnarray}
Moreover, the iteration sequence converges to a feasible point in $C=\bigcap_{i=0}^nC_i$ if the iteration operator $P$ is BPR.
\end{proof}
\begin{remark}
Based on the assumption of iteration $P$, the $\phi$-PP superiorized version of
ART-like converge to a feasible point. In addition the conclusion can be generalized
for EM iteration by imposing some conditions.
\end{remark}
\section{Numerical experiments for XCT image reconstruction}
Although superiorization algorithms can be applied to different image reconstruction methods\cite{CensorBIP,Luo2014,JiangBLT}, we present the application of the
proposed superiorization algorithm to XCT image reconstruction to verify the efficiency of the $\phi$-PP superiorization algorithm.

In our simulations, we used the total variation(TV) function as the objective function $\phi$ \cite{ROF1992,Herman-Acc,Sup_SAP,Pan2006}.
For a $K\times{L}$ digital image $x$, the discrete total variation of $x$ is defined as
\begin{equation}
TV(x)=\sum_{i=1}^{K-1}\sum_{j=1}^{L-1}\sqrt{(x_{i+1,j}-x_{i,j})^2+(x_{i,j+1}-x_{i,j})^2}.
\label{eq10}
\end{equation}
Furthermore, we used the classic ART iteration $P=P_0P_m\cdots P_1$  as the iteration operator $P$ in our numerical experiments.
In order to compare the proposed superiorization algorithm with the classic superiorization algorithm, we applied the classic superiorization and
$\phi$-PP superiorization algorithm to two phantoms (see figure \ref{Fig1}).
The first one is the $200\times200$ Shepp-Logan phantom\cite{Shepp1974}, and the second one is the $256\times256$ head phantom with a ghost which is invisible at 22 specified projection directions \cite{HermanSmall,HermanBook1980}.
In addition, we compare the performances of the two algorithms for the noiseless and noised data with different projections.
In all experiments, the noised projection data was corrupted by additive Gaussian white noise with variance $\sigma^2 = 0.0001$.
We record the iterations,
running time of program and mean square error (MSE)  of different algorithms, where MSE is  computed by
\begin{eqnarray}
MSE(x)=\sqrt{{1\over KL}\sum\limits_{i=1,j=1}^{K,L}\left(x_{ij}-x^0_{ij}\right)^2},
\end{eqnarray}
where $x^0,x$ are the original and estimated images, respectively.

We abbreviate the classic TV-superiorization algorithm as
TV-S, and the proposed algorithm as TV-PPS for convenience.
In the numerical experiments, we used the initial value $x^0=\mathbf{0}$, $\beta_0=10$,
$\gamma=1/2$, and
$\textrm{dist}(x^k,C)=\text{Res}(x)=\|Ax^k-b\|_2$ in algorithm \ref{alg}.
\begin{figure}[H]
\centering
\subfigure[]{
\label{Fig1.sub.1}
\includegraphics[width=4.5cm]{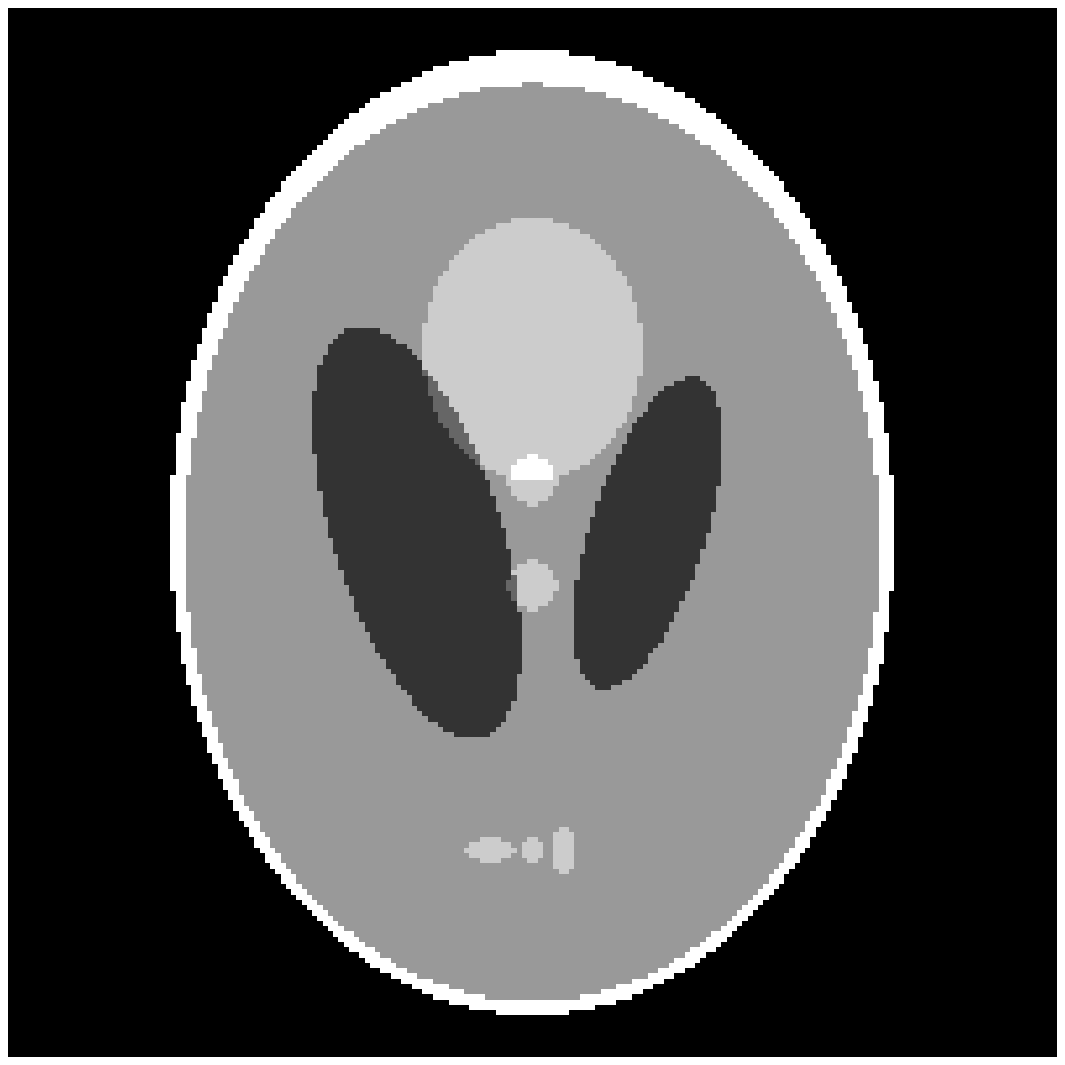}}
\subfigure[]{
\label{Fig1.sub.2}
\includegraphics[width=4.5cm]{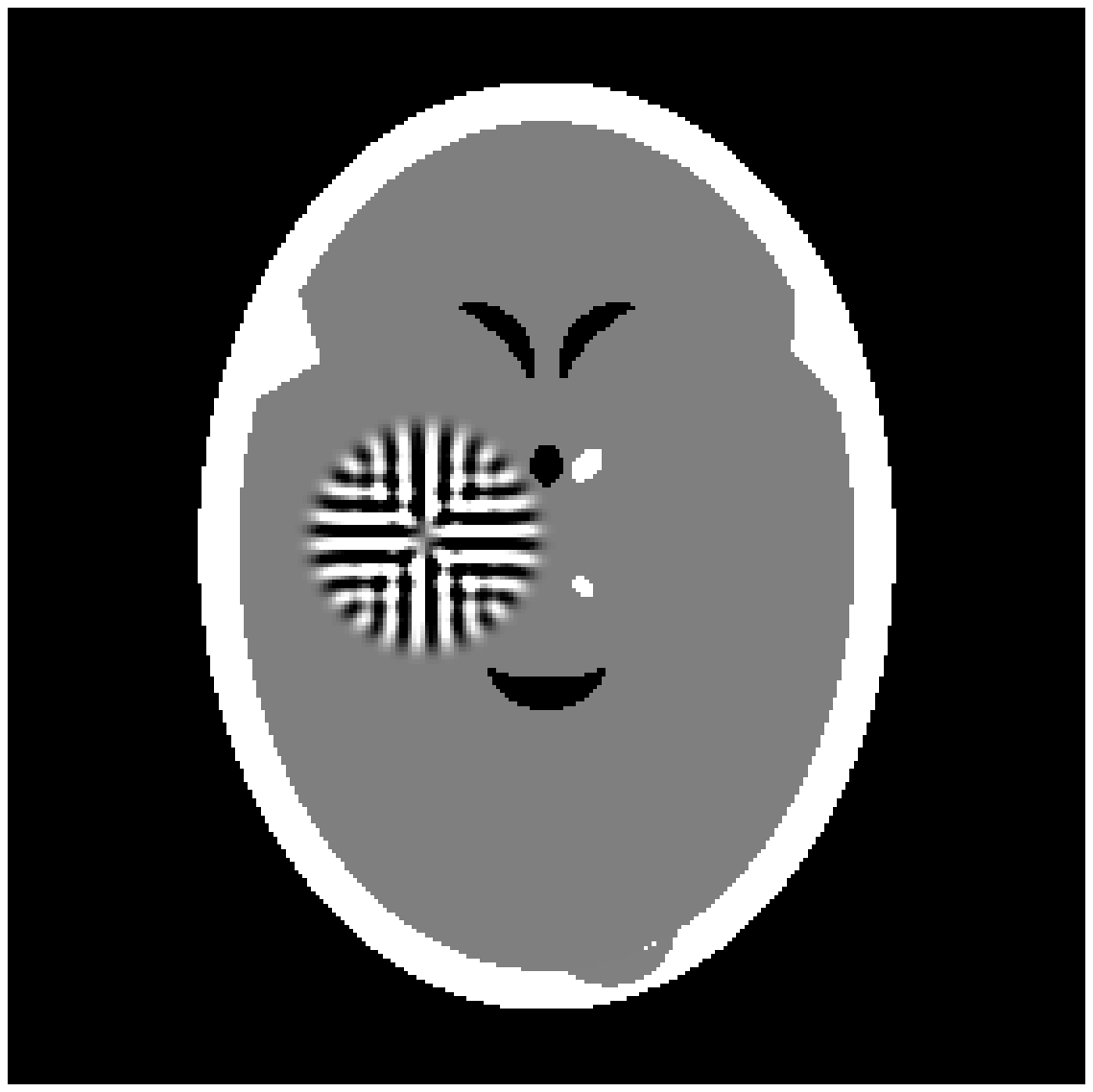}}
\caption{Shepp-Logan phantom (a) and ghost phantom(b).}
\label{Fig1}
\end{figure}
\subsection{Shepp-Logan phantom}
 {\bf Noiseless projection data:}
 The projection data were collected by calculating line integrals across the phantom at  60, 90, 120 directions(equal increments $3^\circ, 2^\circ$ and $1.5^\circ$ from $0^\circ$ to $180^\circ$) of 201 equally spaced parallel lines from $-1$ to $1$.
 {Iteration procedures were terminated when $\textrm{Res}(x^k)\leq{0.01}$ for the noiseless experiments}.

The reconstruction images from the noiseless projection data were shown in the Fig.~\ref{Fig2}.
From Fig.~\ref{Fig2}, we can observe that the classic and the proposed algorithms can
reconstruct images from the three projection data.
In order to show the advantages of the proposed algorithm visually, the central vertical line of the differences  between the reconstructed images and the original image are present in  Fig.~\ref{Fig3}.
We can observe that the $\phi$-PP superiorization is more efficient  than the classic superiorization in the aspect of suppressing the artifacts in the reconstructed images.
\begin{figure}[H]
\centering
\subfigure[TV-S]{
\includegraphics[width=0.3\textwidth]{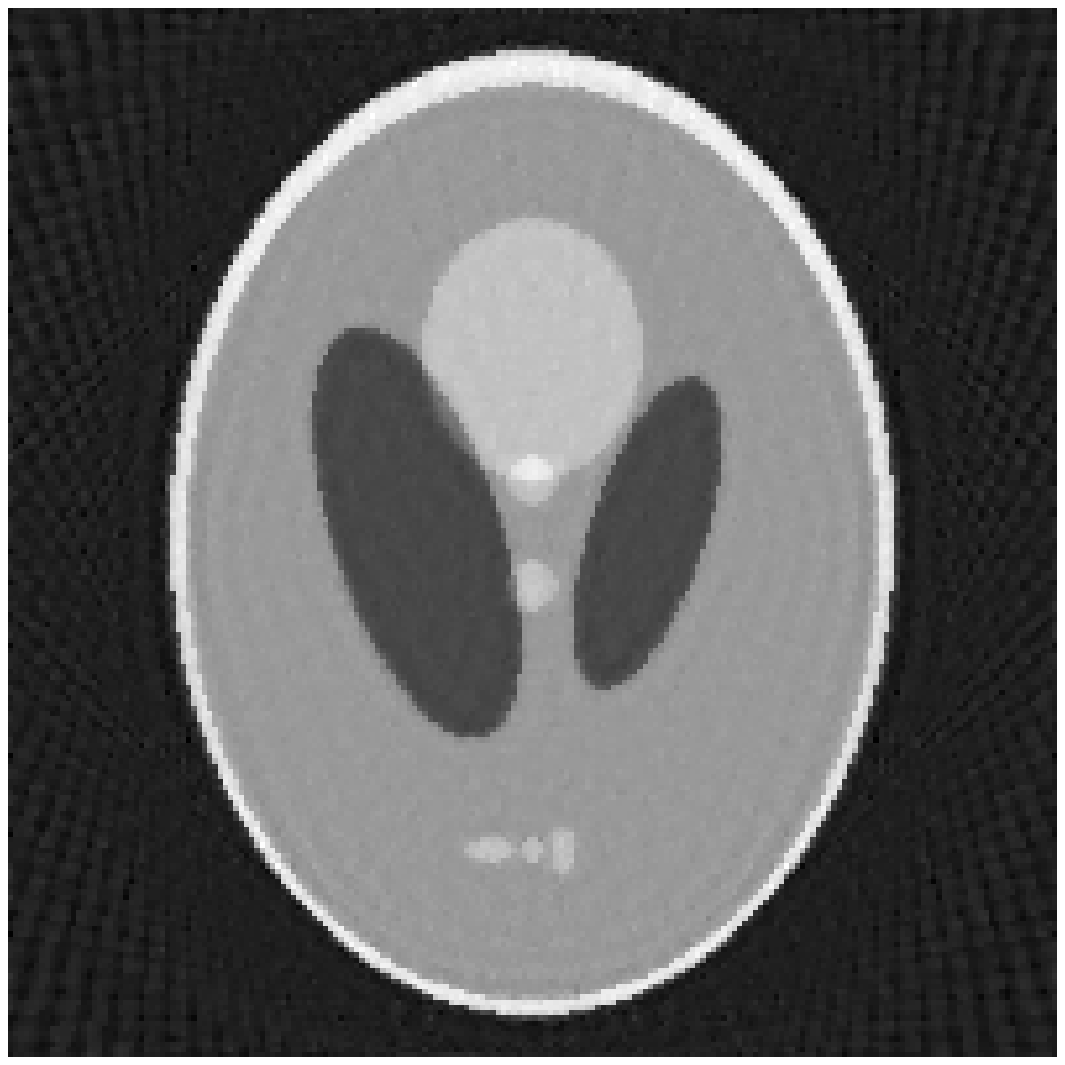}}
\subfigure[TV-S]{
\includegraphics[width=0.3\textwidth]{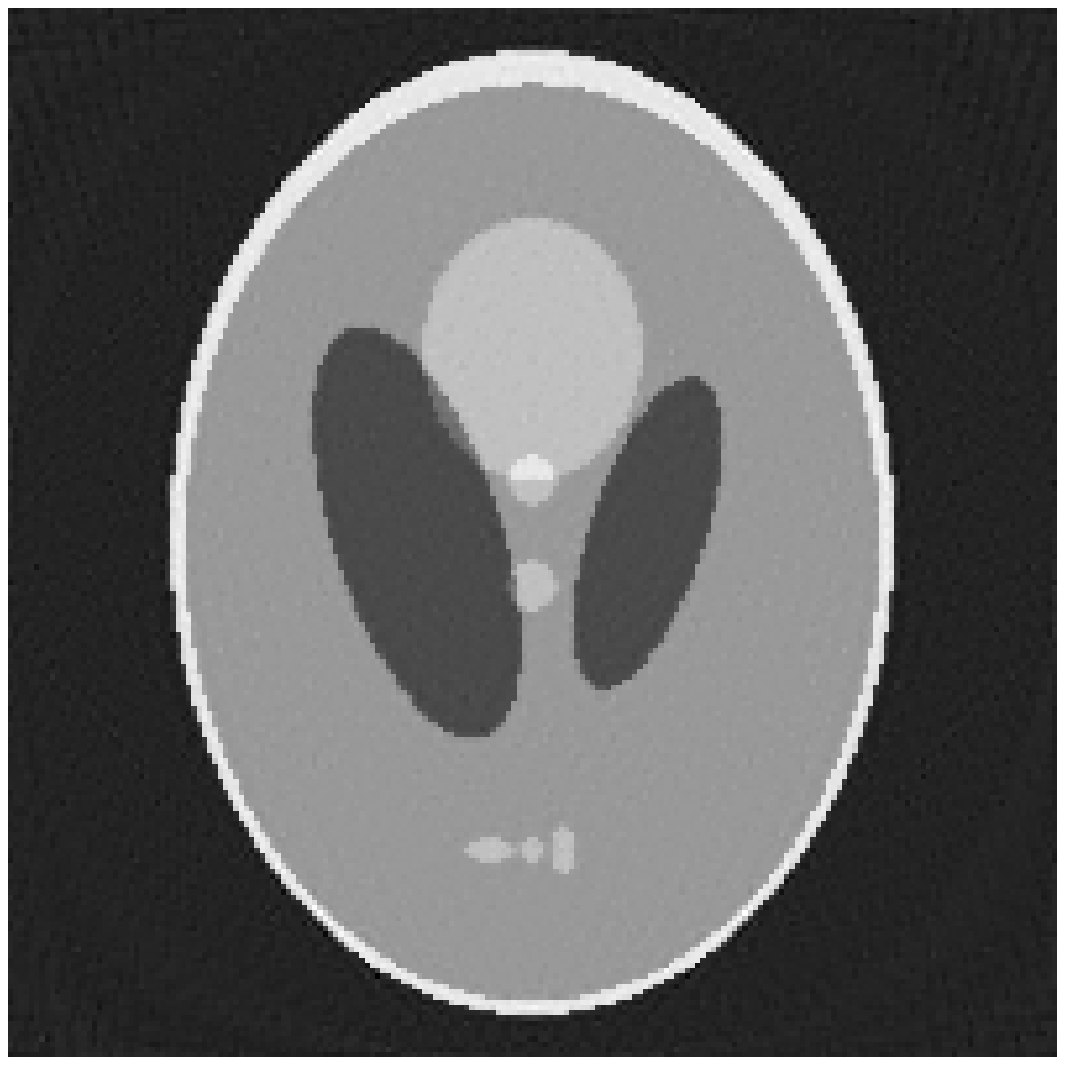}}
\subfigure[TV-S]{
\includegraphics[width=0.3\textwidth]{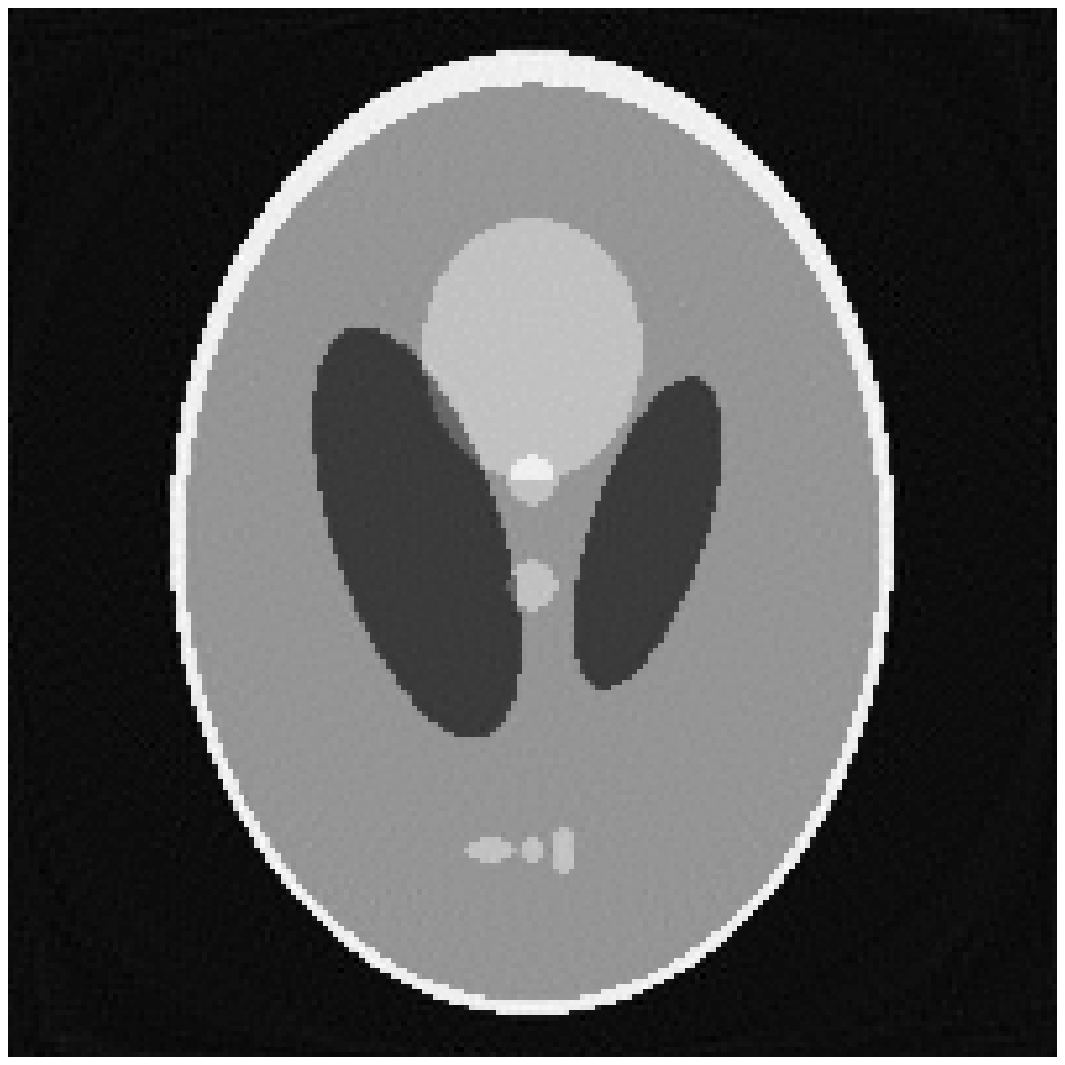}}\\
\subfigure[TV-PPS]{
\includegraphics[width=0.3\textwidth]{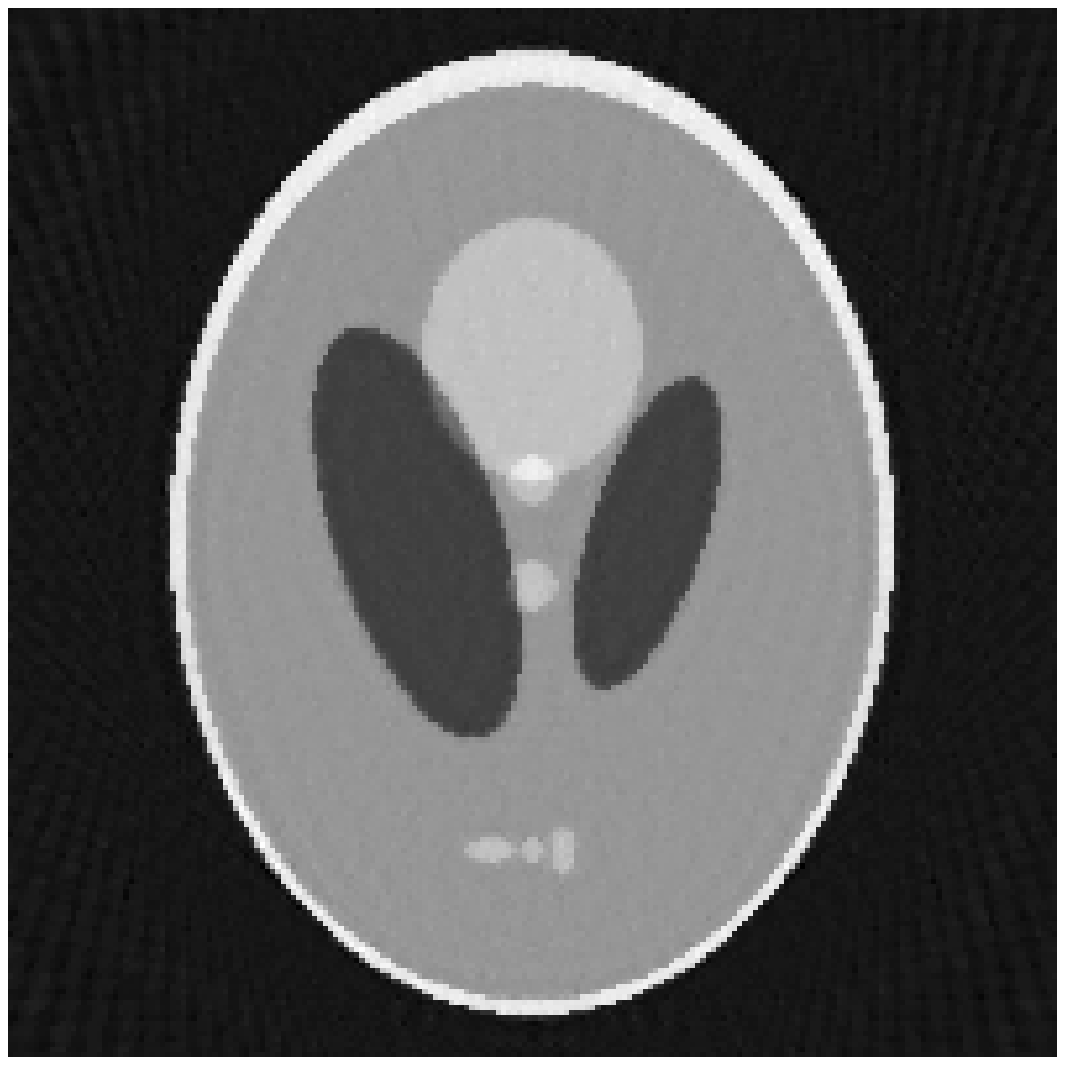}}
\subfigure[TV-PPS]{
\includegraphics[width=0.3\textwidth]{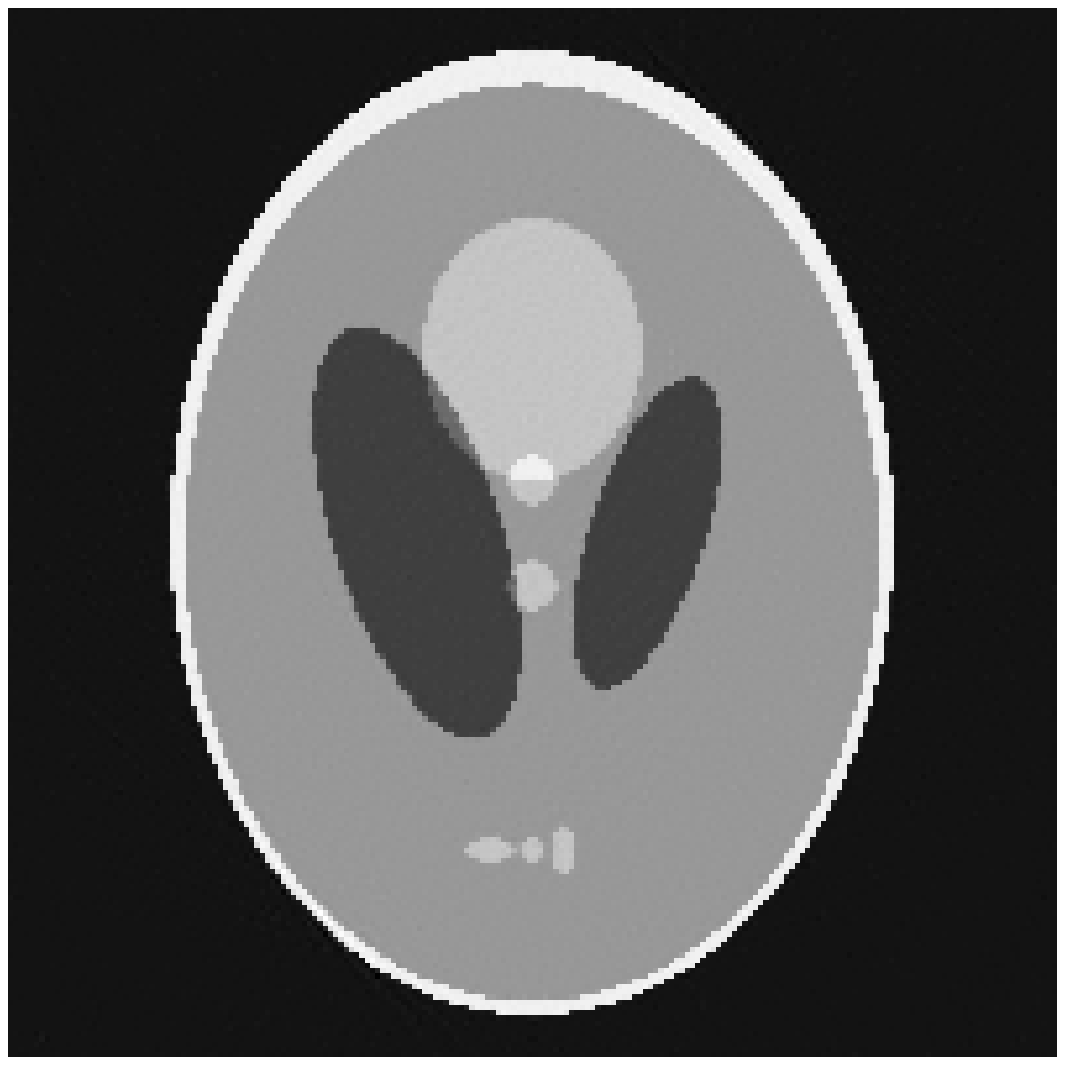}}
\subfigure[TV-PPS]{
\includegraphics[width=0.3\textwidth]{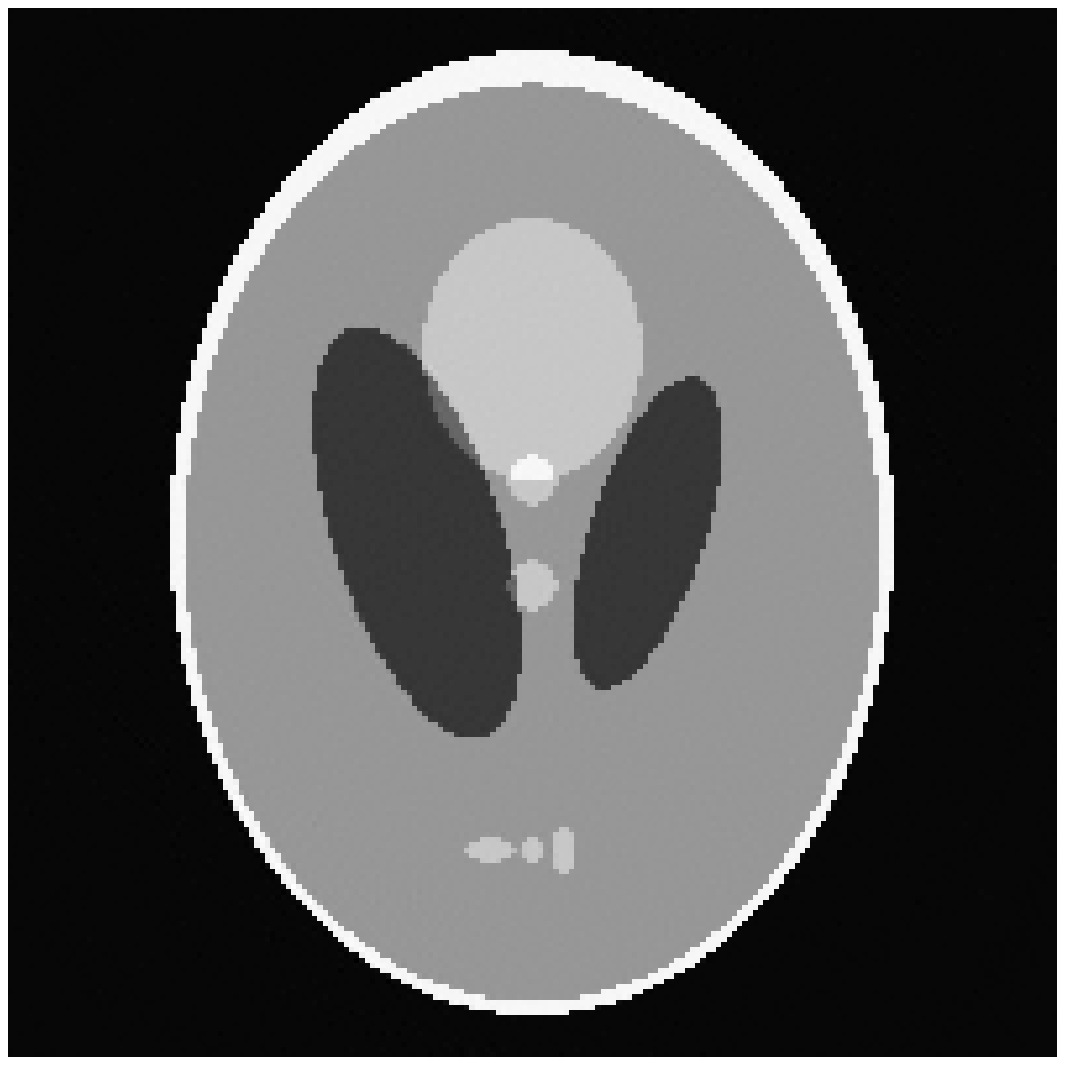}}
\caption{The reconstruction results of Shepp-Logan phantom from 3 noiseless data sets.
The images in first row are reconstructed by TV-S algorithm,
and the images in second row by TV-PPS algorithm.
From left to right, the images in each column are the  reconstruction results from 60, 90, and 120 projections.}
\label{Fig2}
\end{figure}

\begin{figure}[H]
\centering
\begin{tabular}{ccc}
\includegraphics[width=4.5cm]{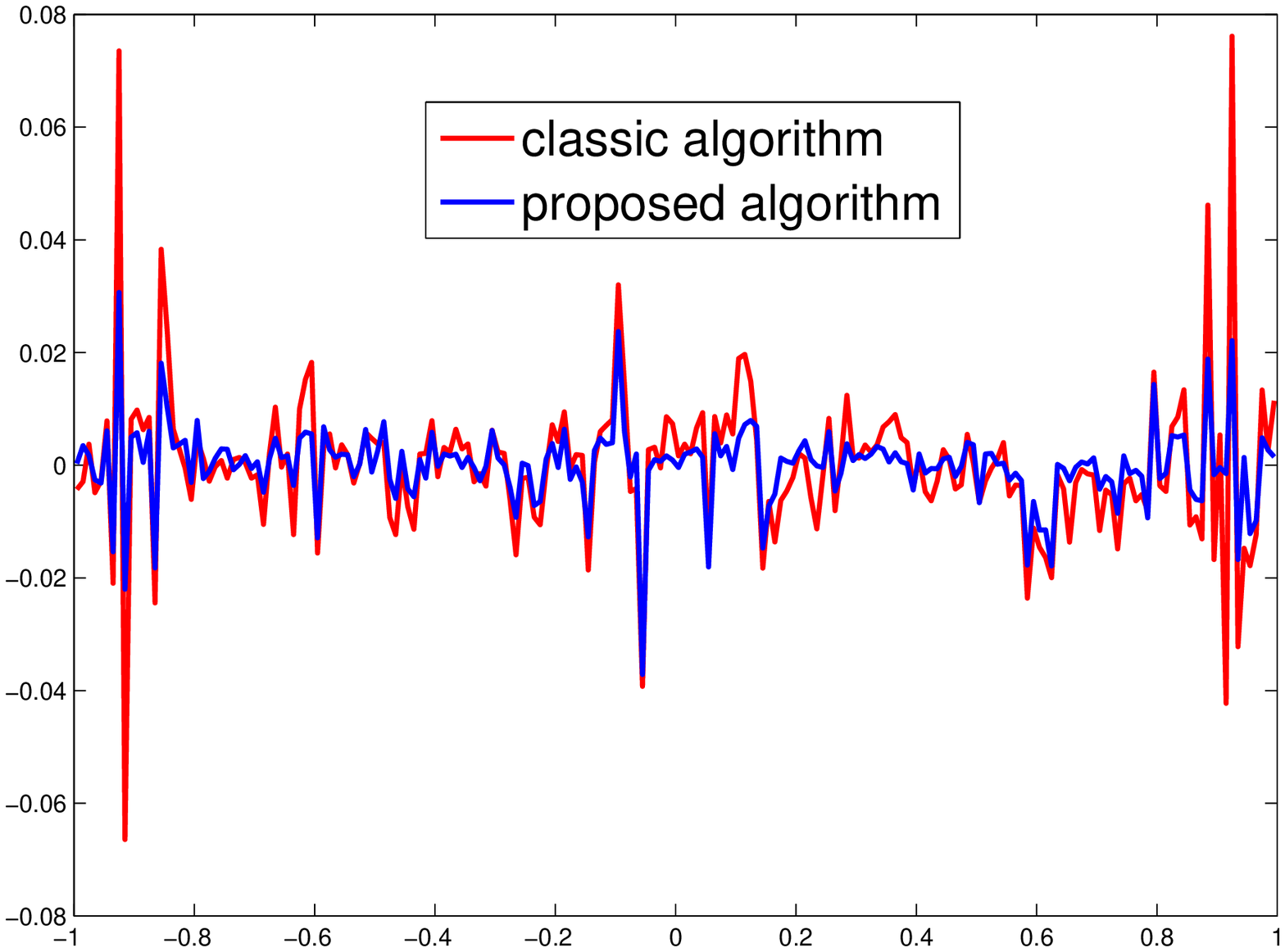}
\includegraphics[width=4.5cm]{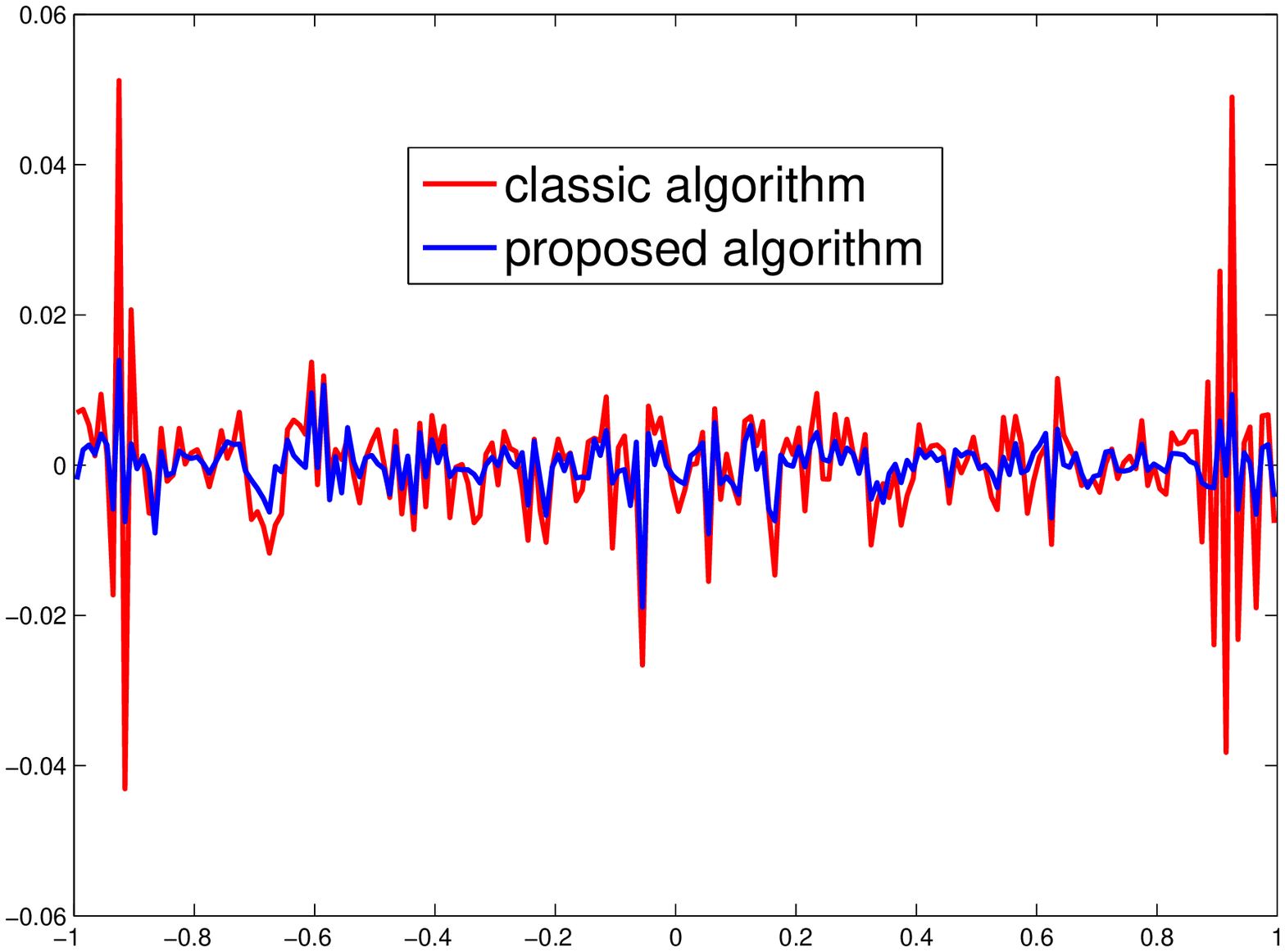}
\includegraphics[width=4.5cm]{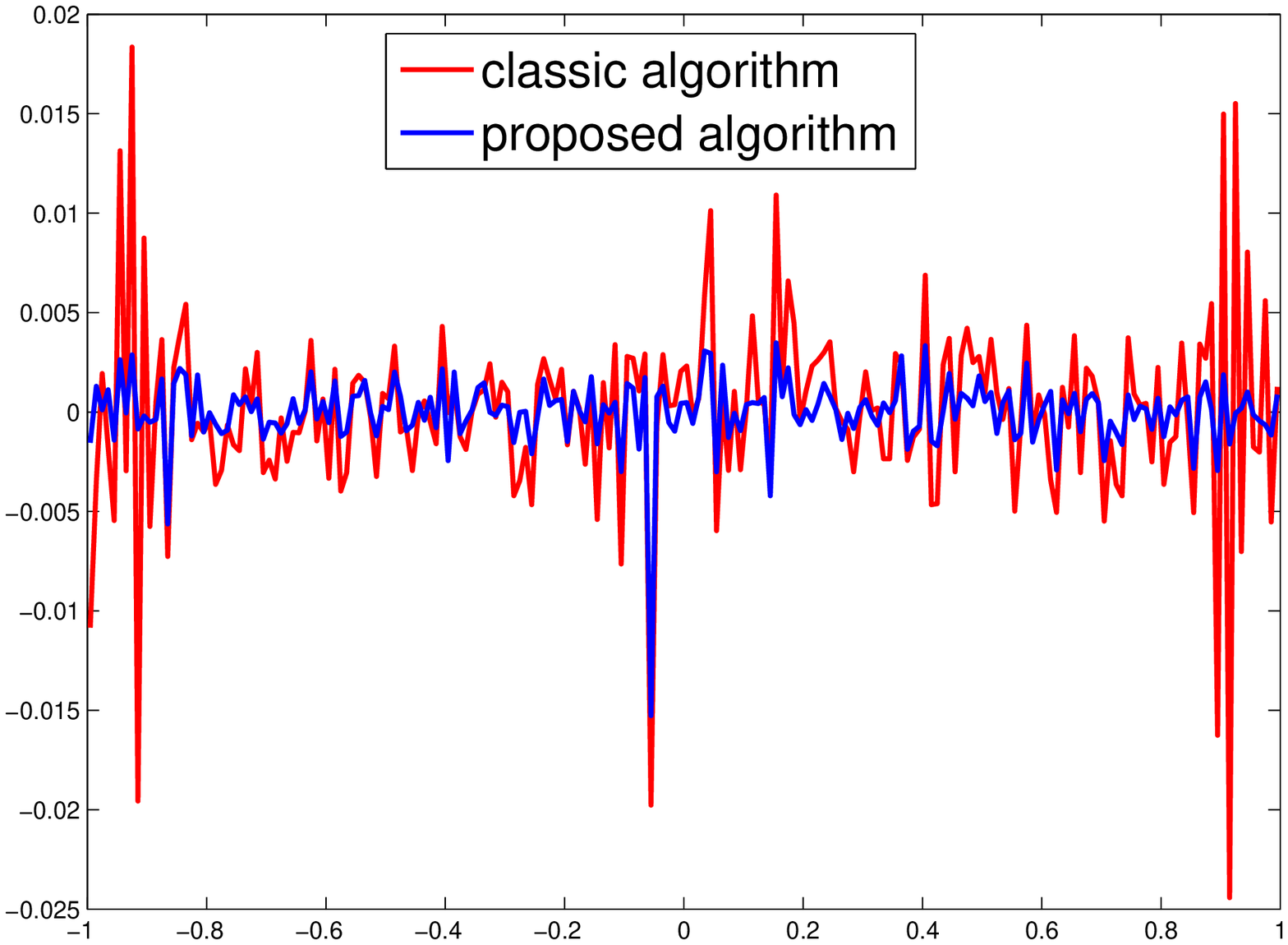}\\
(a)&(b)&(c)
\end{tabular}
\caption{Comparison of profiles. %  off original and reconstructed images.
(a), (b), (c) are the profile differences of the central vertical lines of the images in Fig.~\ref{Fig2} and the original image.}
\label{Fig3}
\end{figure}
In order to compare the images in Fig.~\ref{Fig2} quantitatively,
we tabulated the iterations, MSE, Res and running time(RT) of programs in Table \ref{Tab1}.
By comparing the numbers in Table~\ref{Tab1}, we can draw the conclusion that the proposed method can improve the quality of the estimated images and save computation time.

\begin{table}[H]
\centering
\caption{The MSE, number of iteration and running time(RT) of the images in Fig.~ \ref{Fig2}}
\label{Tab1}
\begin{tabular}{p{2cm}<{\centering}|m{1cm}<{\centering}m{2cm}<{\centering}
m{1cm}<{\centering}m{2cm}<{\centering}m{1cm}<{\centering}m{2cm}<{\centering}m{1cm}
<{\centering}m{1cm}<{\centering}m{1cm}<{\centering}}
\hline
Algorithm   & TV-S  & TV-PPS &TV-S   &TV-PPS  &TV-S    &TV-PPS  \\
\hline
projections       &  120 & 120  & 90   &90      &  60  &60   \\
iterations  &  103 & 97   & 75   & 67    &67    &44  \\
MSEs         &0.0059&0.0022&0.0102& 0.0046&0.0181&0.0097\\
%relative error& 0.0080&0.0023&0.0141&0.0061&  0.0255& 0.0109\\
RT(min)   &15.179&11.3611&5.5021 & 5.3384& 4.4305 & 3.0206 \\
\hline
\end{tabular}
\end{table}
\begin{figure}[H]
\centering
\includegraphics[width=0.3\textwidth]{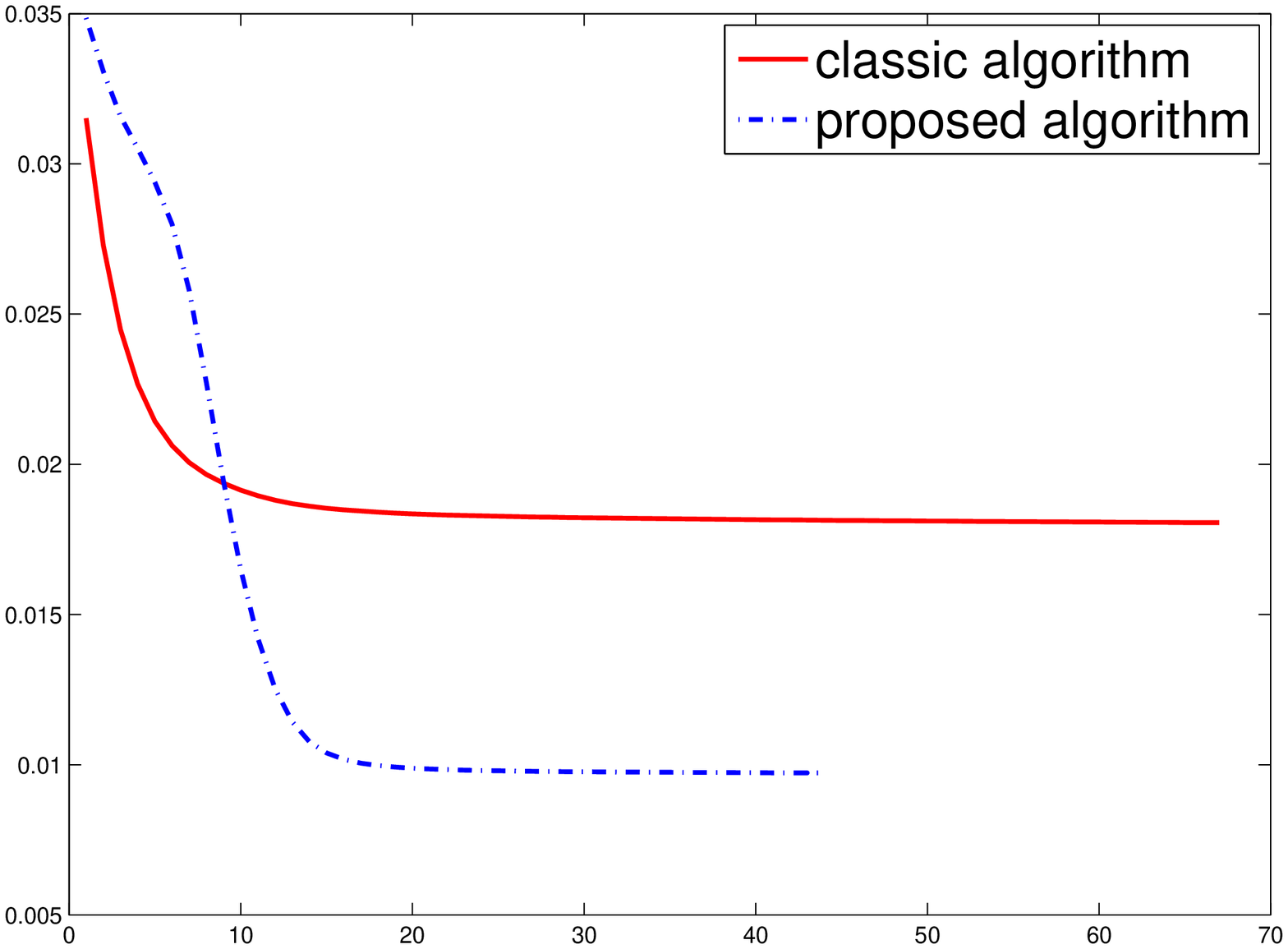}
\includegraphics[width=0.3\textwidth]{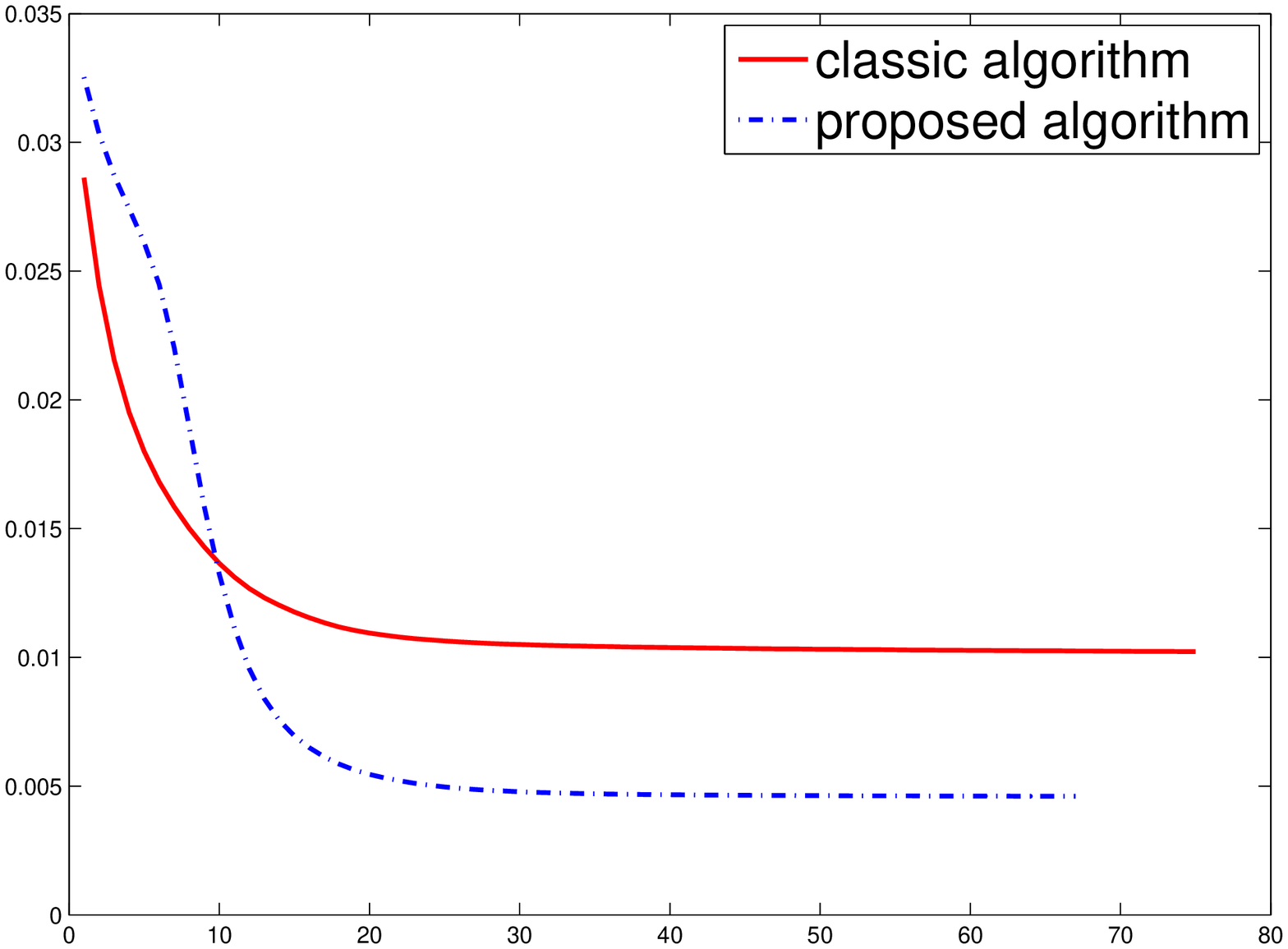}
\includegraphics[width=0.3\textwidth]{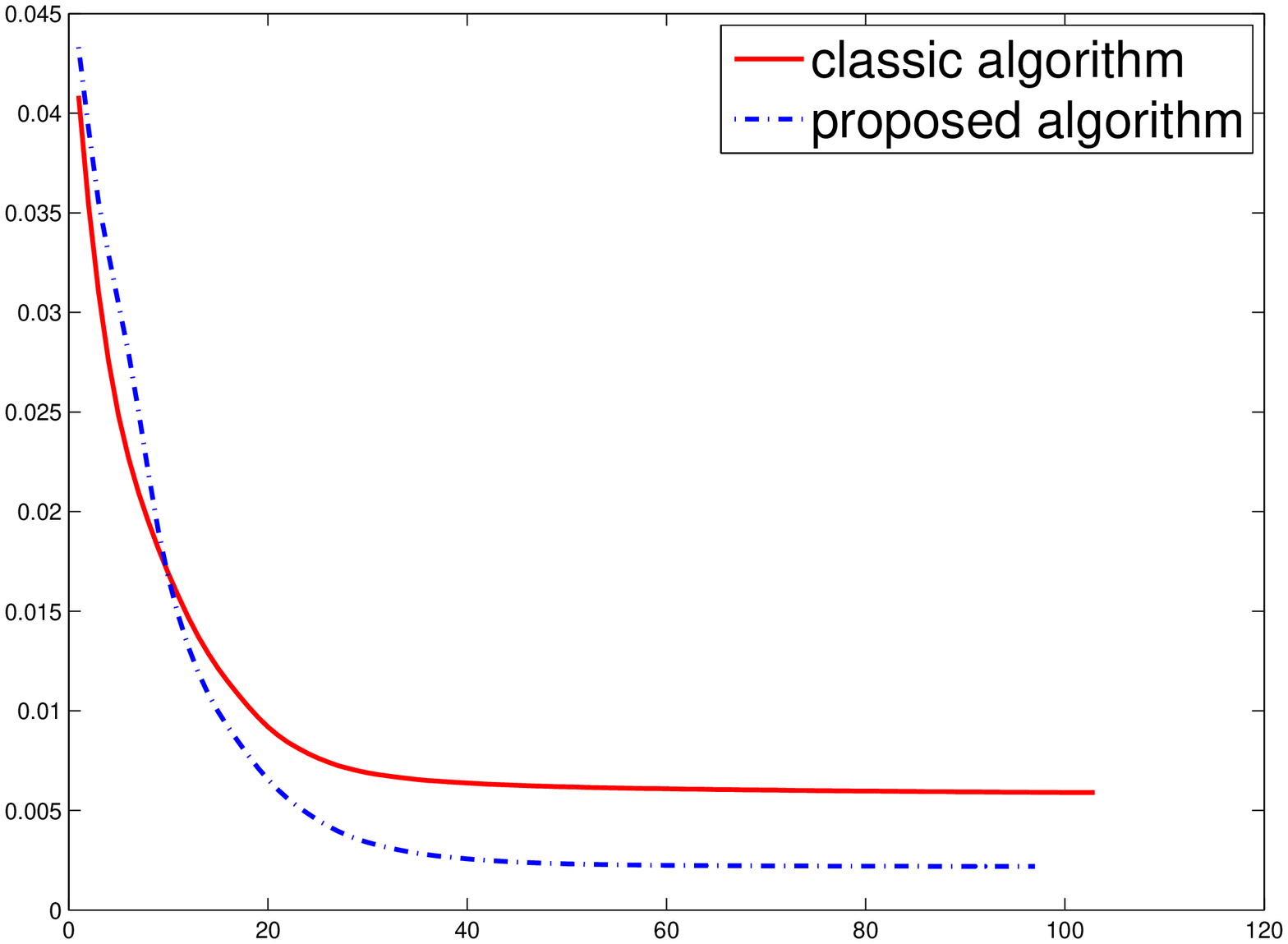}
\caption{MSE vs. iteration number of the two algorithms for Shepp-Logan phantom, from left to right for 60, 90, and 120 projections.}
\label{Fig4}
\end{figure}
In order to compare the convergent speed of the proposed algorithms with
 the classic algorithms visually,
we present the evolution of MSE along with the iteration process
in Fig.~\ref{Fig4}  for the 3 projection data. And we can observe that the proposed perturbation
can accelerate the convergent rate and improve the reconstructed image qualities.

{\bf Noised projection data:}
In order to show the ability of noise suppressing, we apply the algorithms to noised projection data. For the noise experiments, the iteration procedures were terminated the algorithms  when $\textrm{Res}(x^k)\leq{0.1},0.12,0.14$ for  60, 90 and 120 projections.
The reconstruction images were given in Fig.~\ref{Fig5}, and
Table~\ref{Tab2} showed the MSE, number of iterations, running time of different images in Fig.~\ref{Fig5}.

\begin{figure}[H]
\centering
\subfigure[TVS]{
\includegraphics[width=0.3\textwidth]{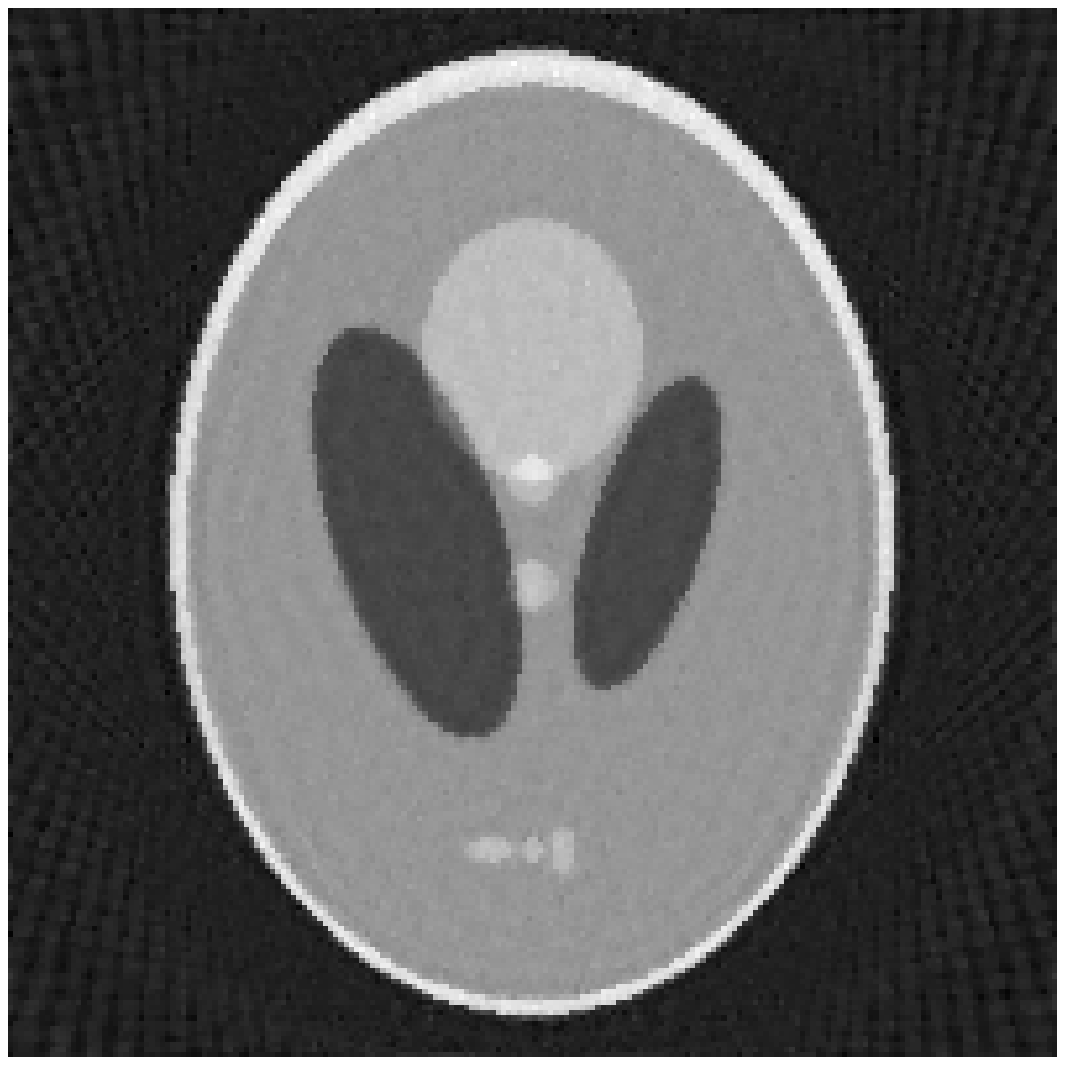}}
\subfigure[TVS]{
\includegraphics[width=0.3\textwidth]{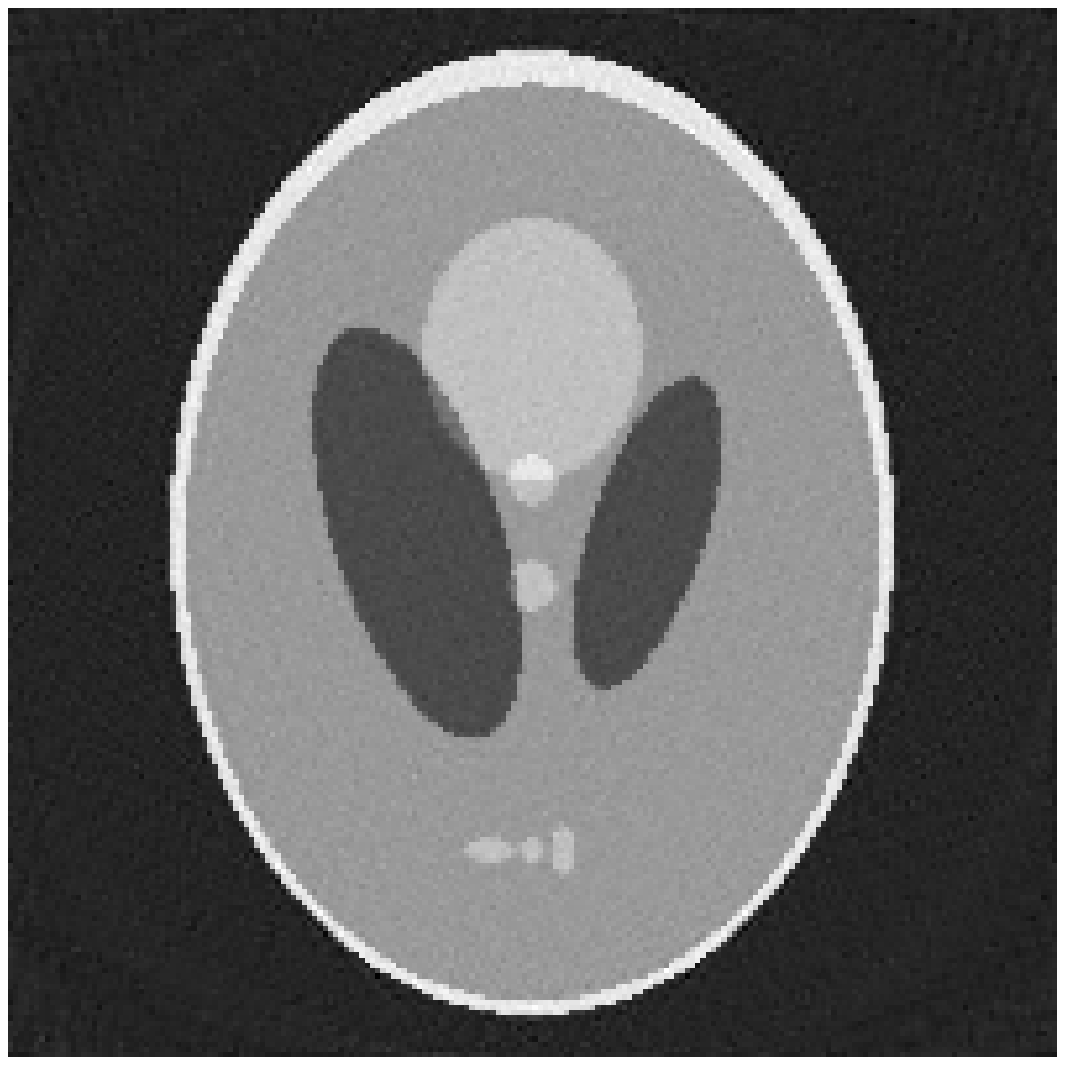}}
\subfigure[TVS]{
\includegraphics[width=0.3\textwidth]{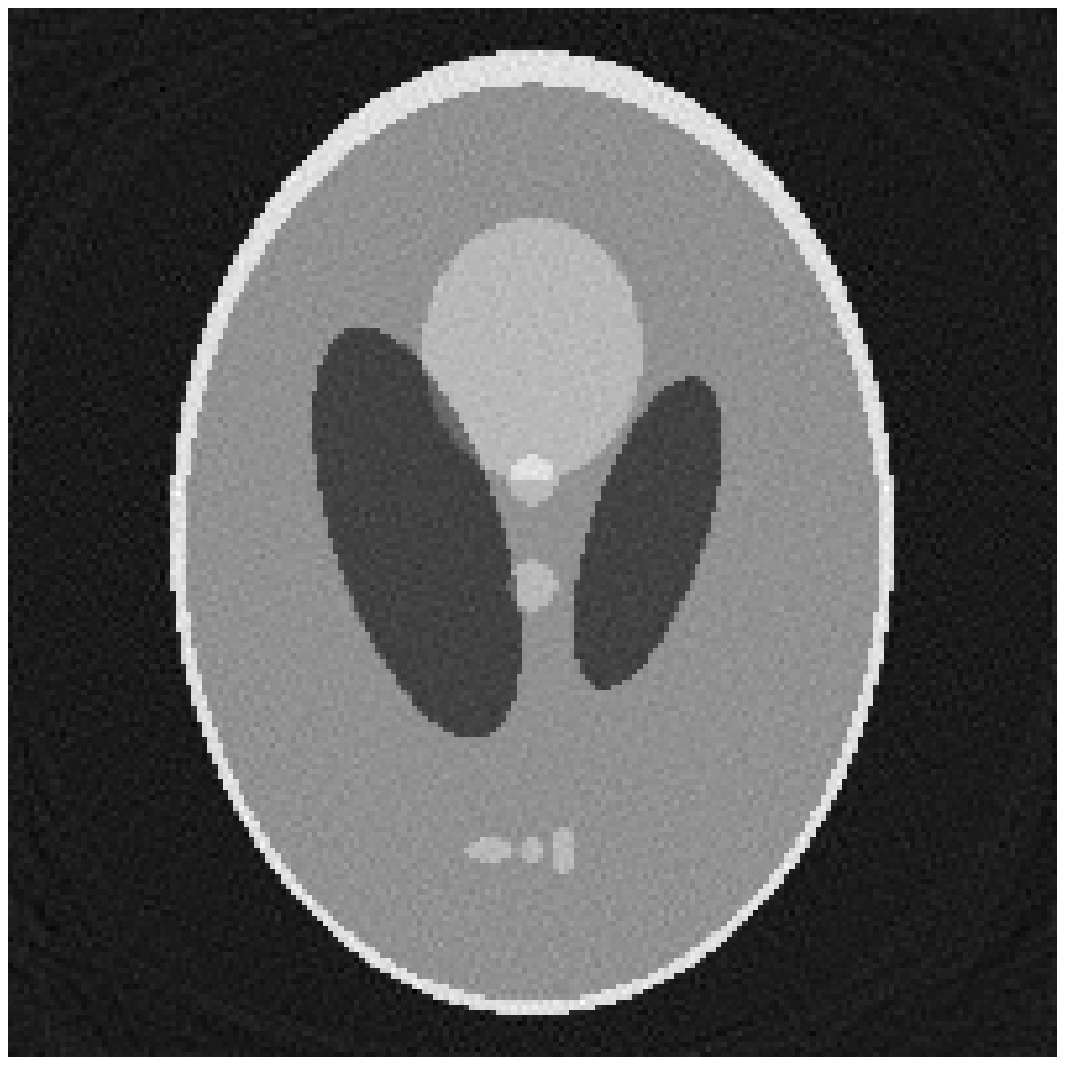}}\\
\subfigure[TV-PPS]{
\includegraphics[width=0.3\textwidth]{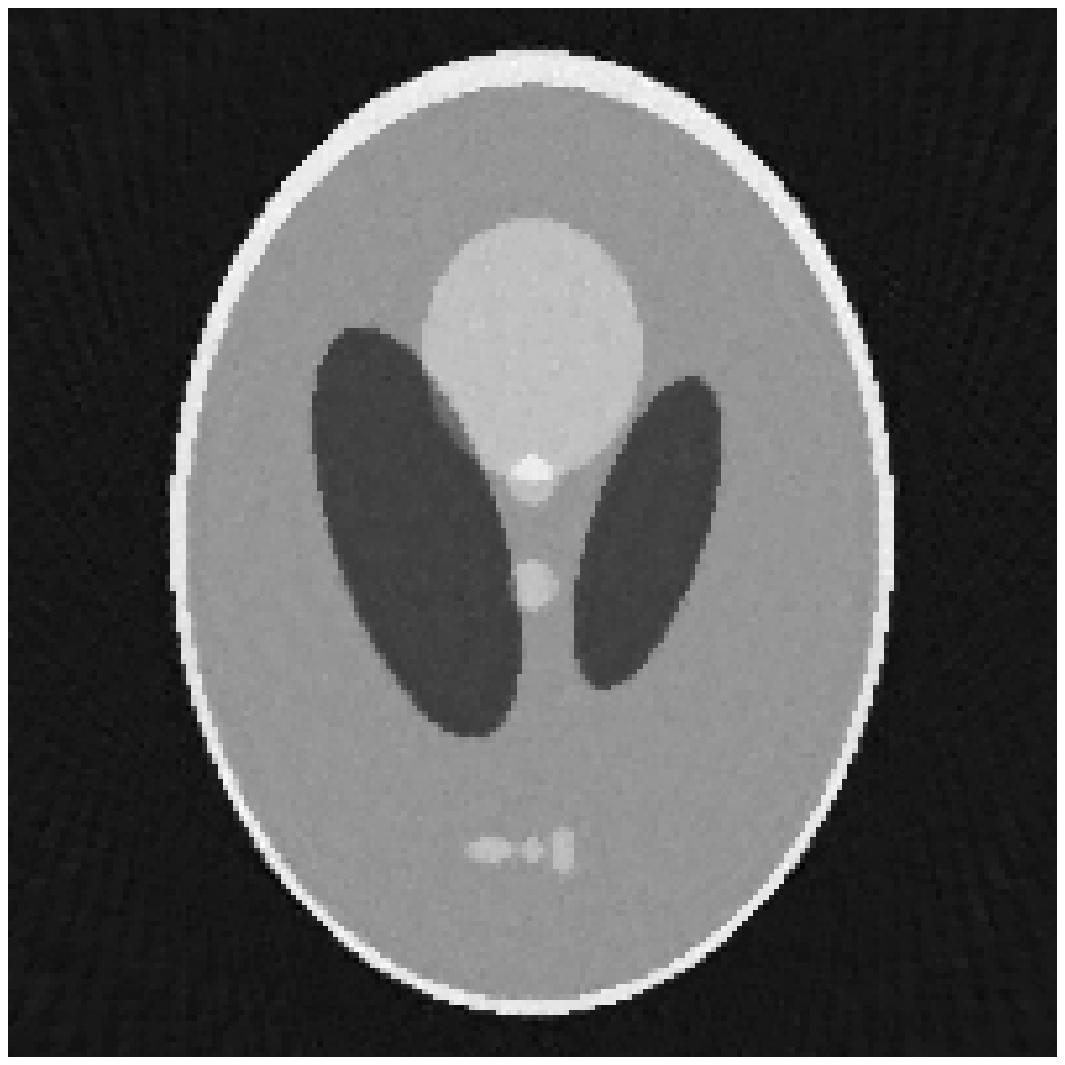}}
\subfigure[TV-PPS]{
\includegraphics[width=0.3\textwidth]{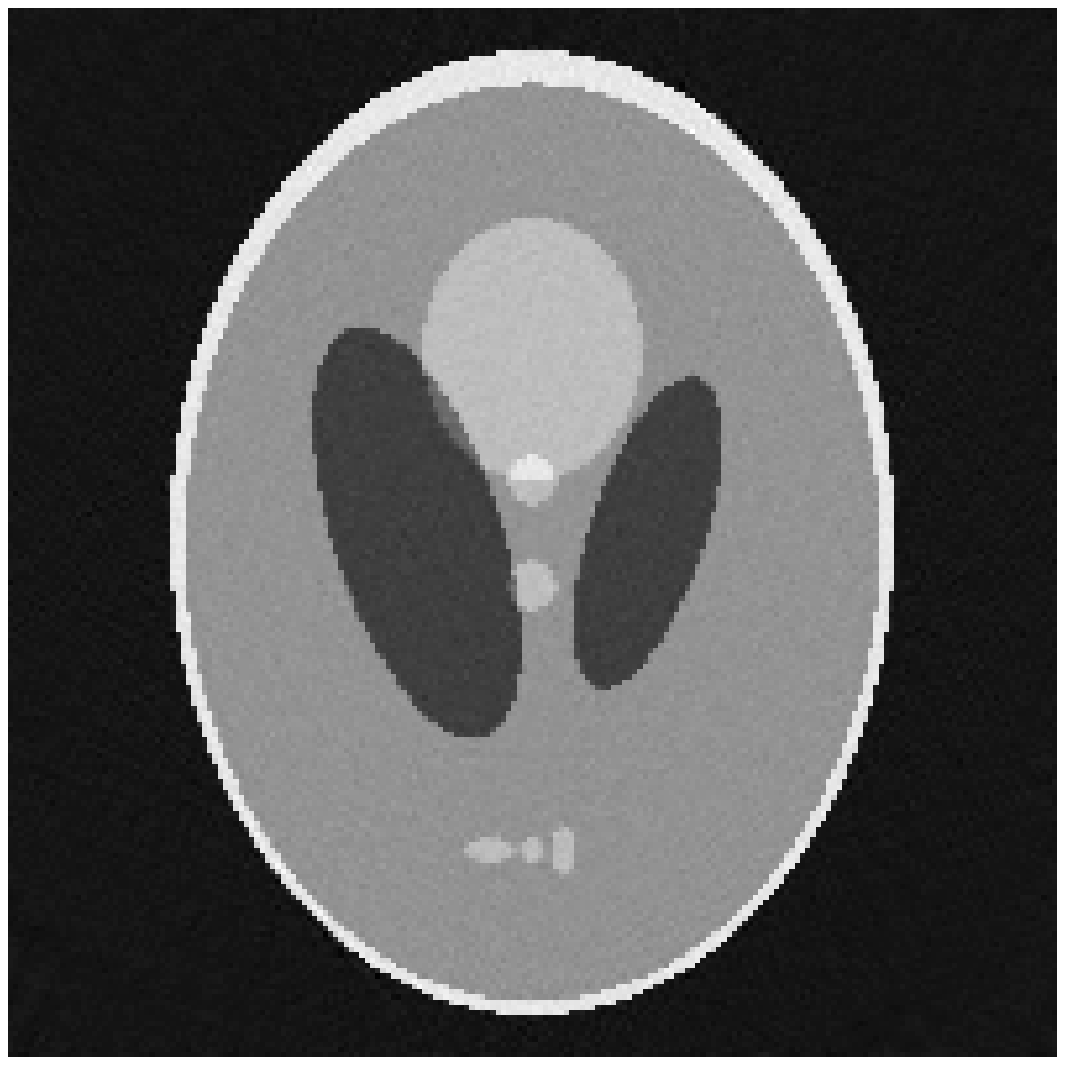}}
\subfigure[TV-PPS]{
\includegraphics[width=0.3\textwidth]{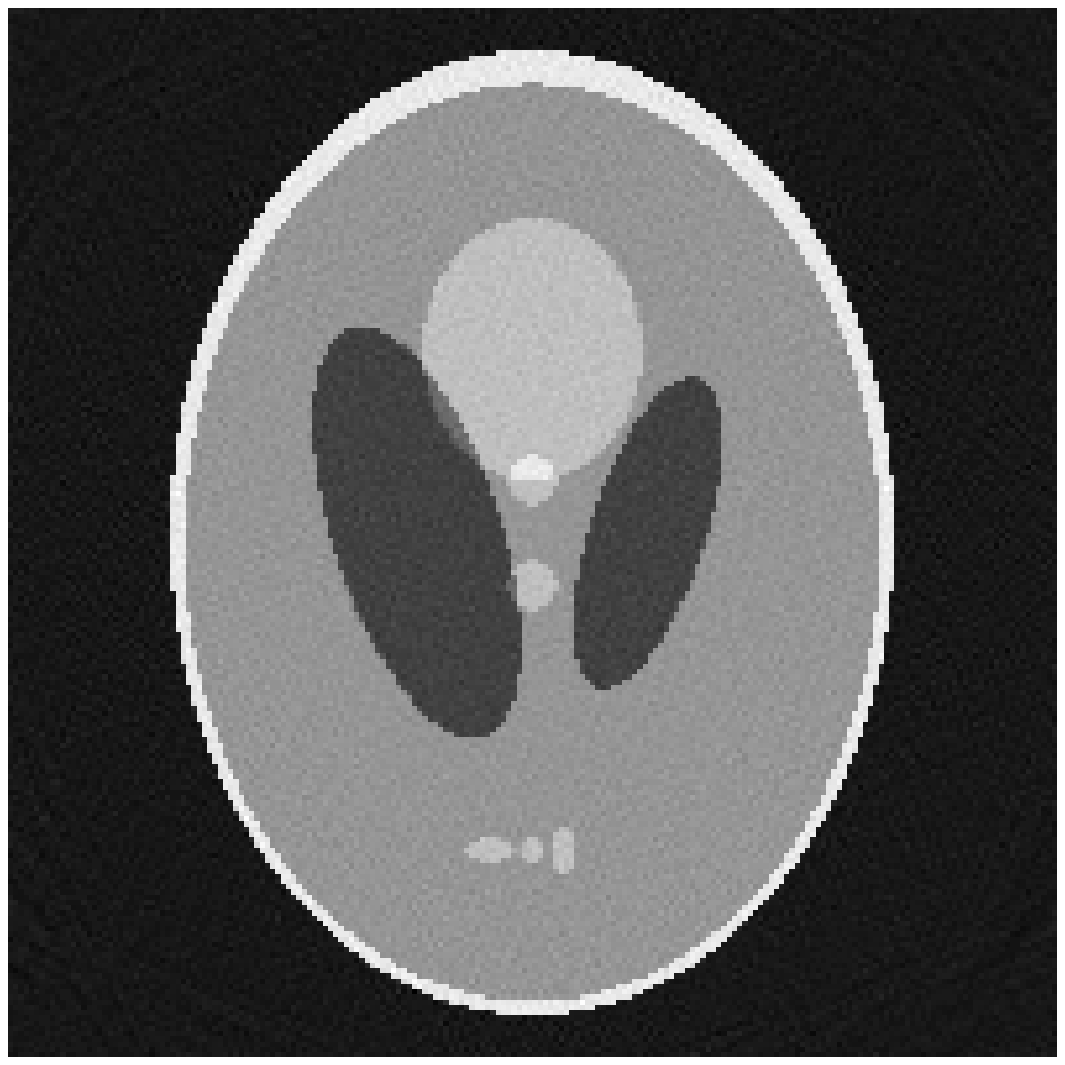}}
\caption{The reconstruction results of Shepp-Logan phantom from noised projections.
 From left to right, the images in each column are reconstructed from 60, 90, and 120 projections.}
\label{Fig5}
\end{figure}

\begin{table}[H]
\centering
\caption{The MSE, number of iteration and running time(RT) of images in Fig.~ \ref{Fig5}}
\label{Tab2}
\begin{tabular}{p{2cm}<{\centering}|m{1cm}<{\centering}m{2cm}<{\centering}
m{1cm}<{\centering}m{2cm}<{\centering}m{1cm}<{\centering}m{2cm}<{\centering}m{1cm}
<{\centering}m{1cm}<{\centering}m{1cm}<{\centering}}
\hline
Algorithm   & TV-S  & TV-PPS &TV-S   &TV-PPS  &TV-S    &TV-PPS  \\
\hline
projections       &  120 & 120  & 90   &90      &  60  &60   \\
iterations  &  49 & 49   & 34  & 33      &25    &24  \\
MSEs         &0.0134&0.0108&0.0132& 0.0085 &0.0192&0.0112\\
RT(min)   &5.6175&6.6477&2.8230& 2.5867&1.5444&1.4879 \\
\hline
\end{tabular}
\end{table}

By comparing Fig.~\ref{Fig5} and Table~\ref{Tab2}, we can observe the reconstructed images by the proposed superiorization algorithm have higher quality than these by classic superiorization algorithm. %Table \ref{Tab2} showed the modified superiorization algorithm has a smallest MSE and the fastest reconstruction speed among the two algorithms.
Therefore, the performance of the proposed superiorization algorithm is better
than the classic superiorized algorithms for noised projection data.

In summary, the proposed superiorization algorithm has faster reconstruction speed than the classic superiorization algorithms, and the MSEs of the reconstruction images by the proposed algorithm are
 smaller than these of the reconstructed images by the classic superiorization algorithm regardless noiseless and noised projection data.
Therefore, the above results demonstrate that the modified superiorization algorithm can accelerate the convergent speed and improve the image qualities.% visually and quantatively.

\subsection{Ghost phantom}
{\bf Noiseless projection data:}
 Since the ghost in this phantom is invisible at 22  directions \cite{HermanSmall, Sup_BIP},
 the reconstruction images usually suffer from artifacts. in our simulations,
 the projection data were collected in 112 and 82 directions:
 90 and 60  with equal angle increments from $0^\circ$ to $179^\circ$ and 22 specified views in which the ghost is invisible \cite{HermanSmall}.  {Iteration procedures were terminated when $\textrm{Res}(x^k)\leq{0.01}$ for the noiseless projections}.

The reconstruction images from the noiseless projection data were shown in the Fig.~\ref{Fig6}. For comparison, Table~\ref{Tab3} present the iterations, MSE, Res and running time(RT) of
different reconstruction results.

\begin{figure}[H]
\centering
\begin{tabular}{cc}
\includegraphics[width=4.5cm]{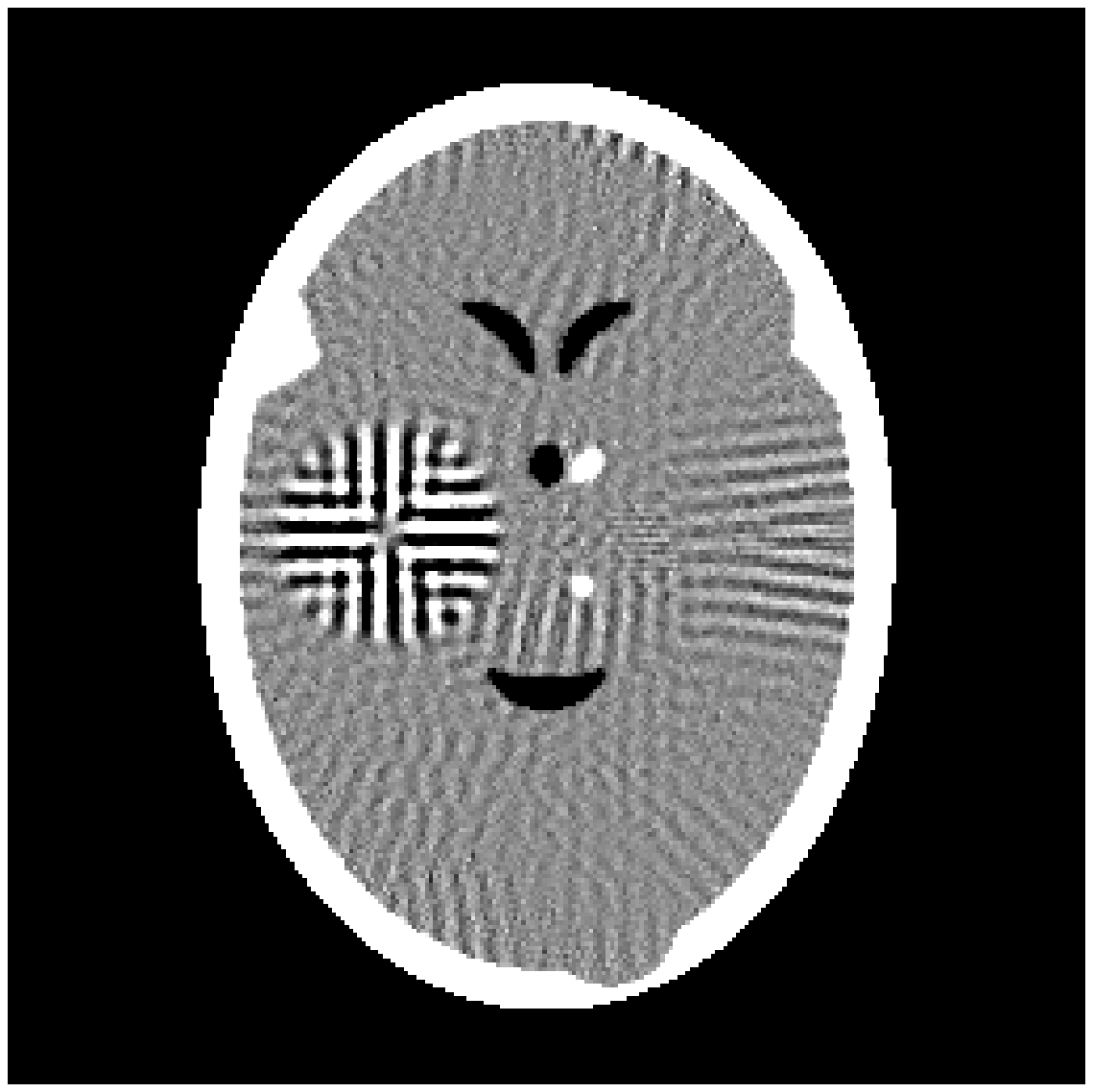}&
\includegraphics[width=4.5cm]{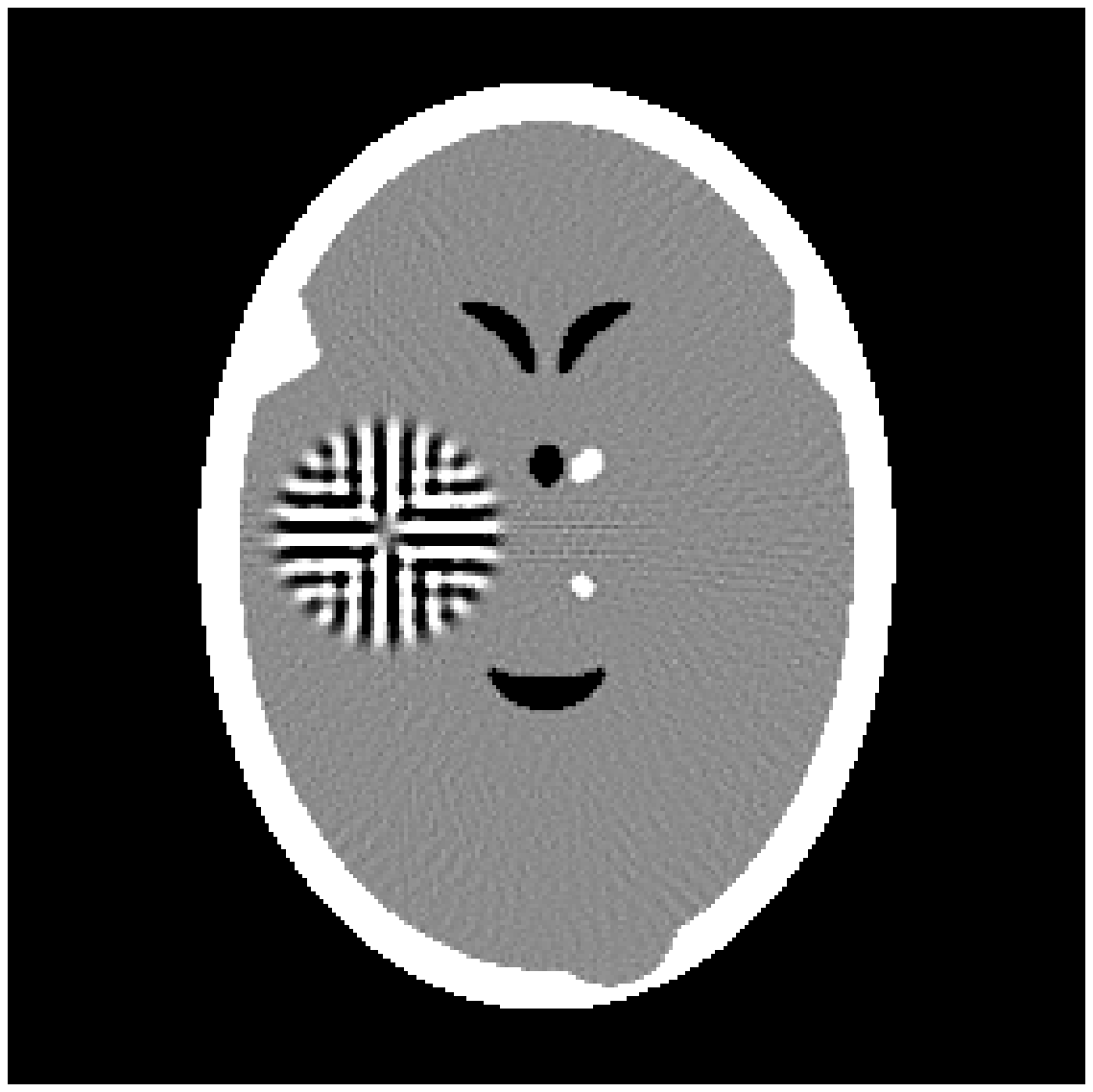}\\
TV-S&TV-S
\\
\includegraphics[width=4.5cm]{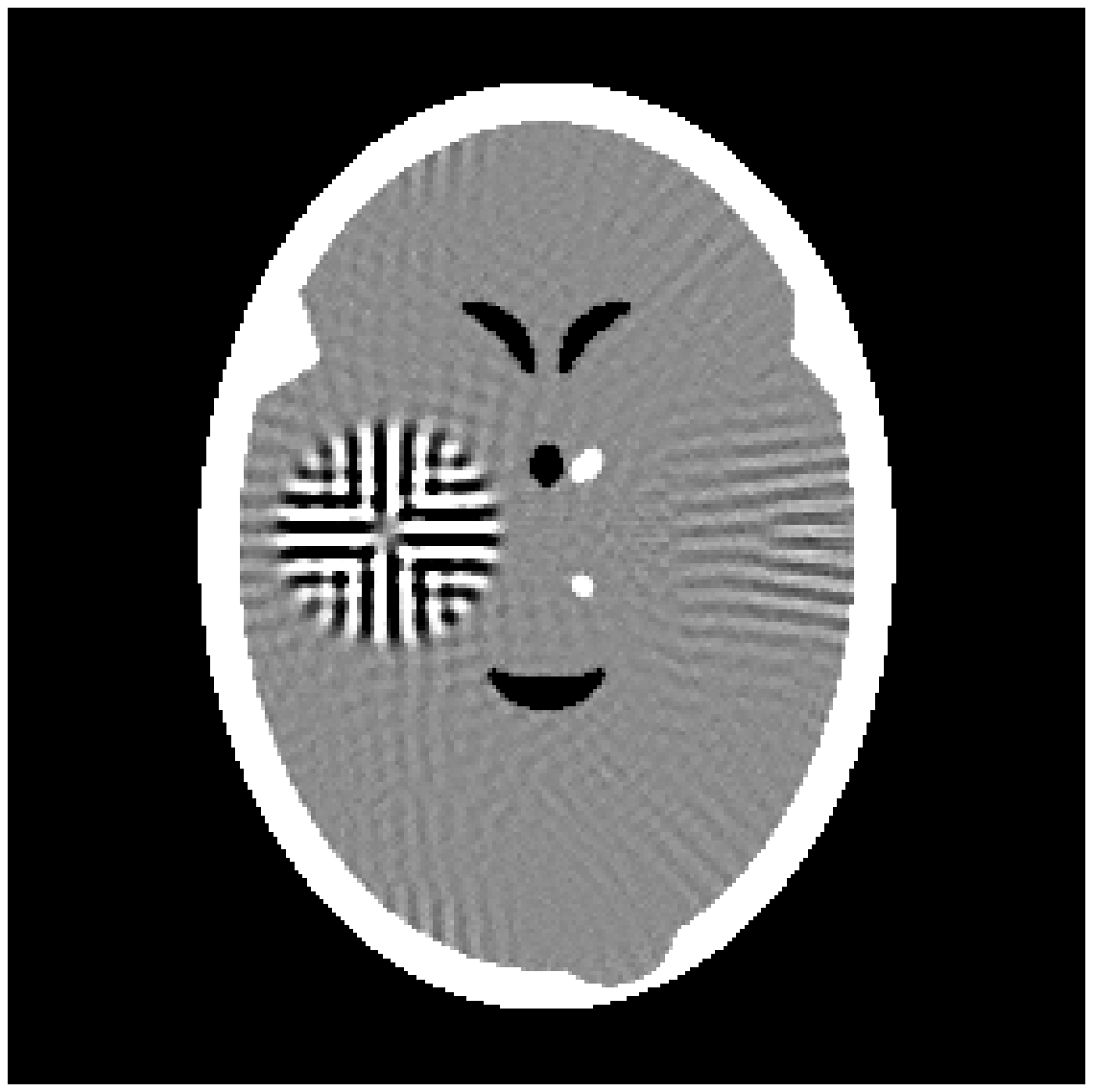}&
\includegraphics[width=4.5cm]{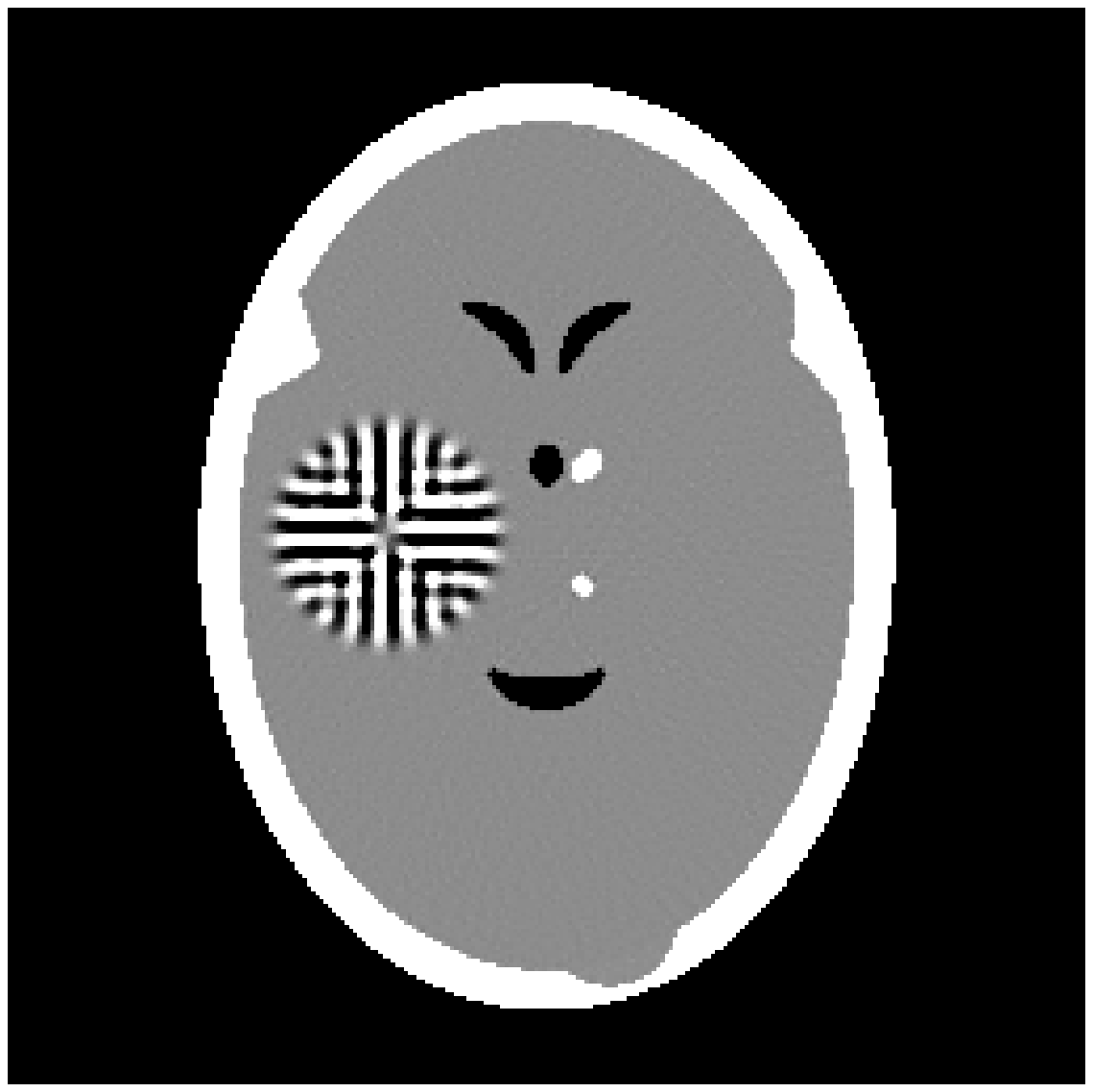}\\
TV-PPS&TV-PPS
\end{tabular}
\caption{The reconstruction results of ghost phantom from noiseless projections.
From left to right, images in each column are reconstructed from 82, 112 projections .}
\label{Fig6}
\end{figure}
\begin{table}[H]
\centering
\caption{The MSE, number of iteration and running time(RT) of the images in Fig.~ \ref{Fig6}}
\label{Tab3}
\begin{tabular}{p{2cm}<{\centering}|m{1cm}<{\centering}m{2cm}<{\centering}
m{1cm}<{\centering}m{2cm}<{\centering}m{1cm}<{\centering}m{2cm}<{\centering}m{1cm}
<{\centering}m{1cm}<{\centering}m{1cm}<{\centering}}
\hline
Algorithm   & TV-S  & TV-PPS &TV-S  &TV-PPS\\%  &TVS    &MTVS \\
\hline
projections       & 112  &112    &  82  & 82 \\%  &60    &60\\
iterations  & 24   &  19   &33    &32  \\%  &27    &25\\
MSEs         &0.0056& 0.0026&0.0108&0.0083\\%&0.0159&0.0119\\
RT(min)   &16.82 &13.89&10.83 &10.87 \\%&10.50 &10.1894\\
\hline
\end{tabular}
\end{table}

{\bf Noised projection data:}
For the noised projection data,  the iteration processes were terminated when $\textrm{Res}(x^k)\leq 0.15,0.13$  for 82 and 112 projections. The reconstruction images were given in Fig.~\ref{Fig7}. Table~\ref{Tab4} showed the MSEs, iterations and running time of program of the results of images in Fig.~\ref{Fig7}.

\begin{figure}[H]
\centering
\begin{tabular}{cc}
\includegraphics[width=4cm]{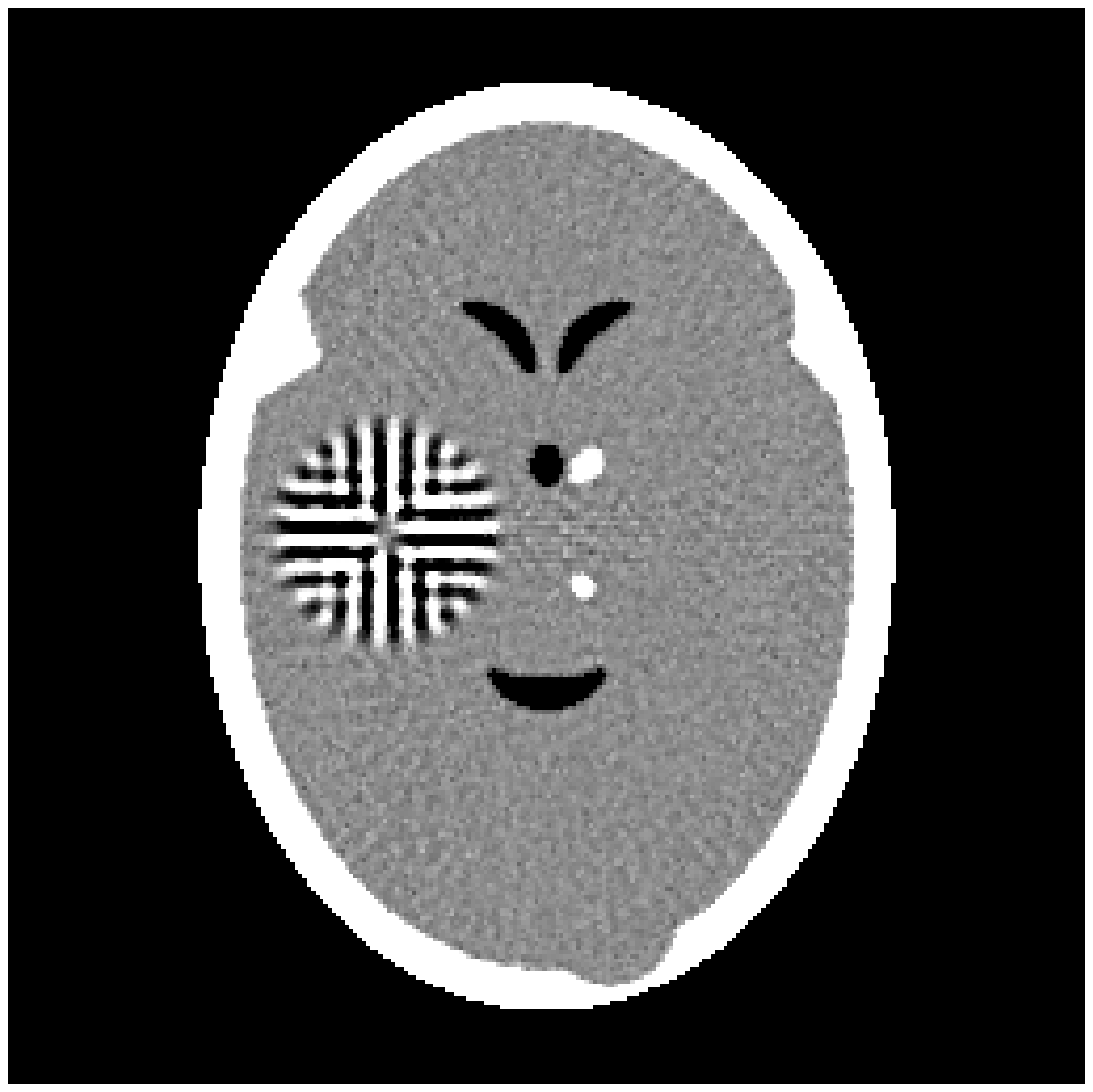}&
\includegraphics[width=4cm]{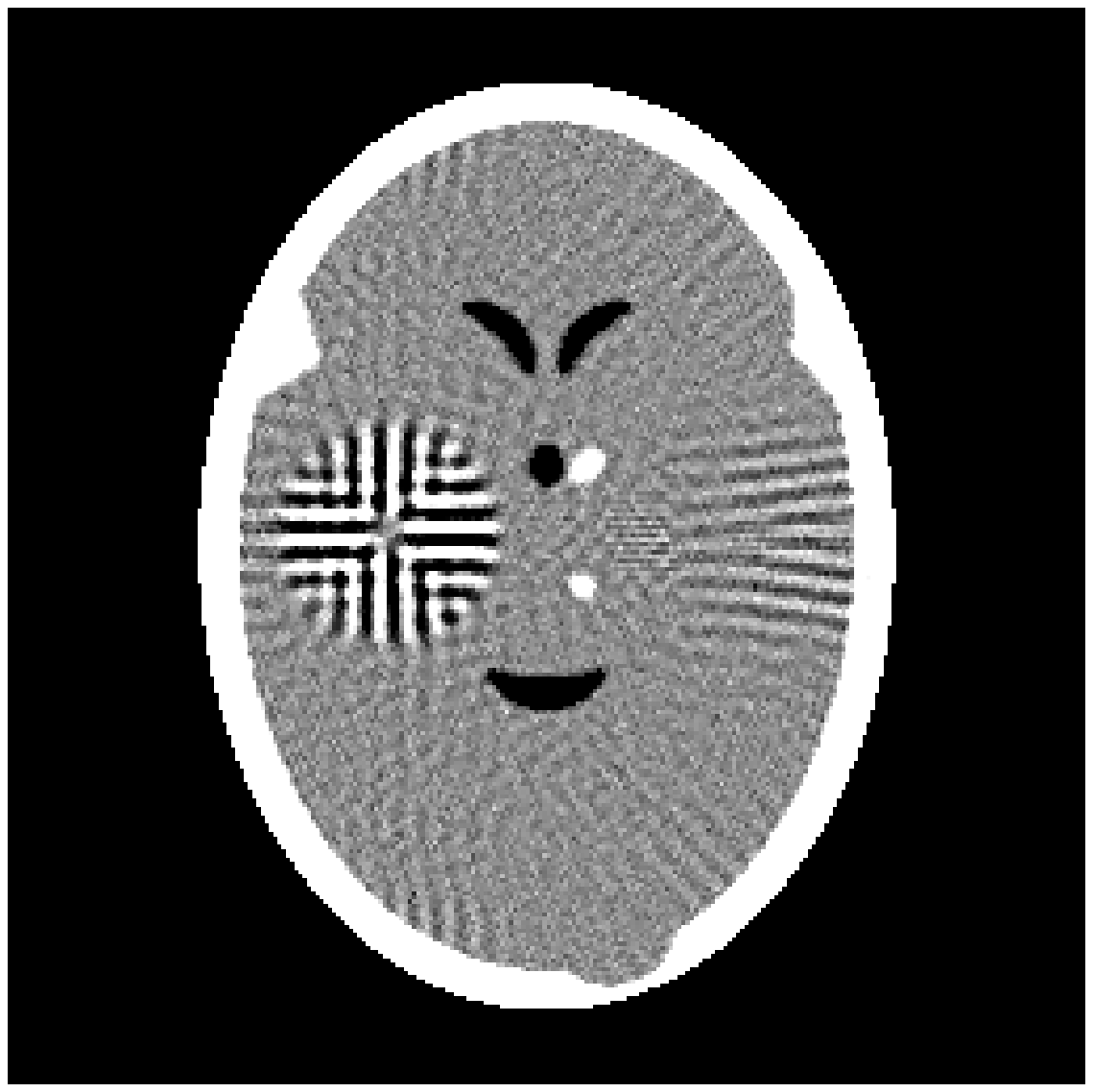}
\\
TV-S&TV-S\\
\includegraphics[width=4cm]{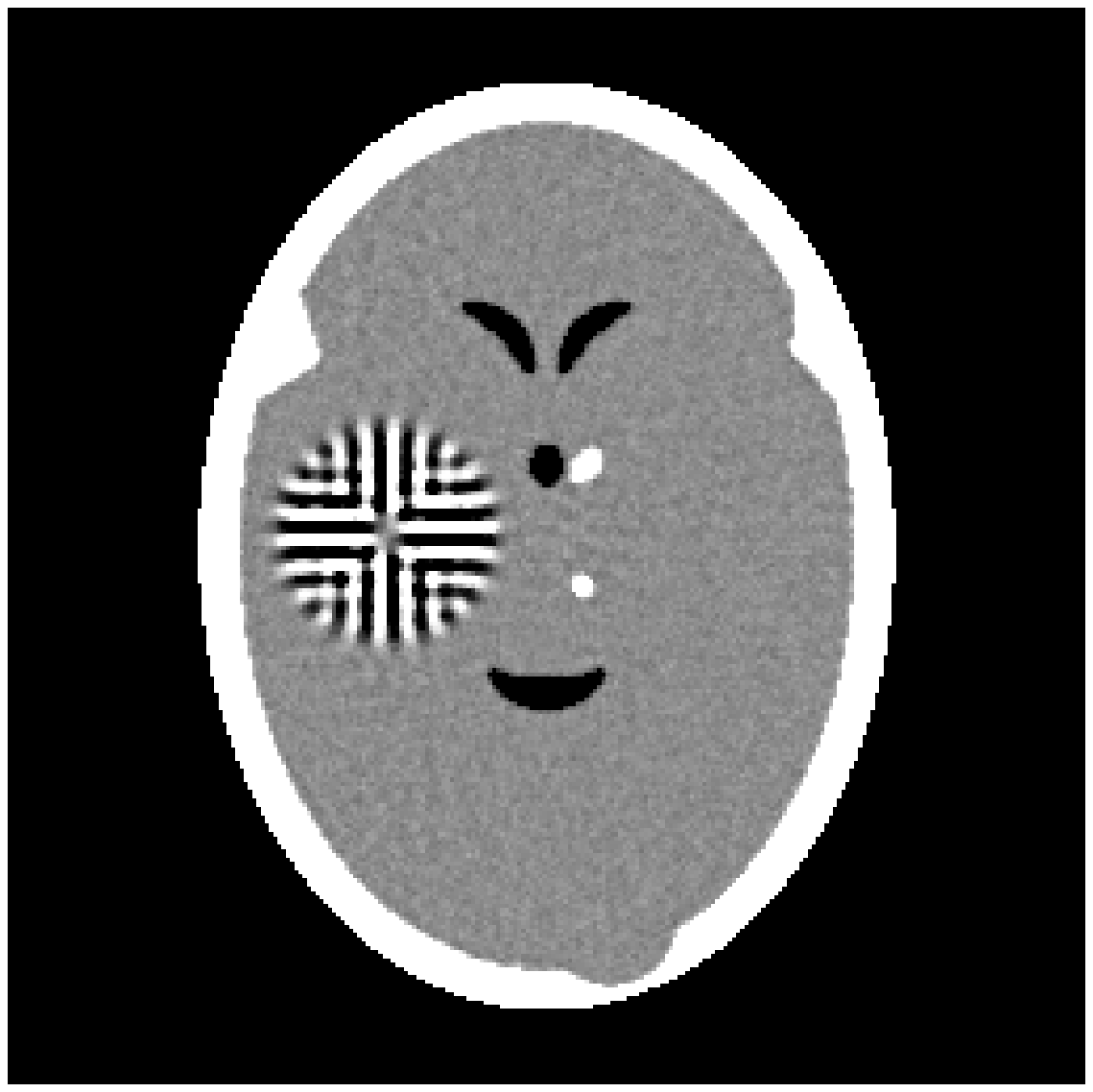}&
\includegraphics[width=4cm]{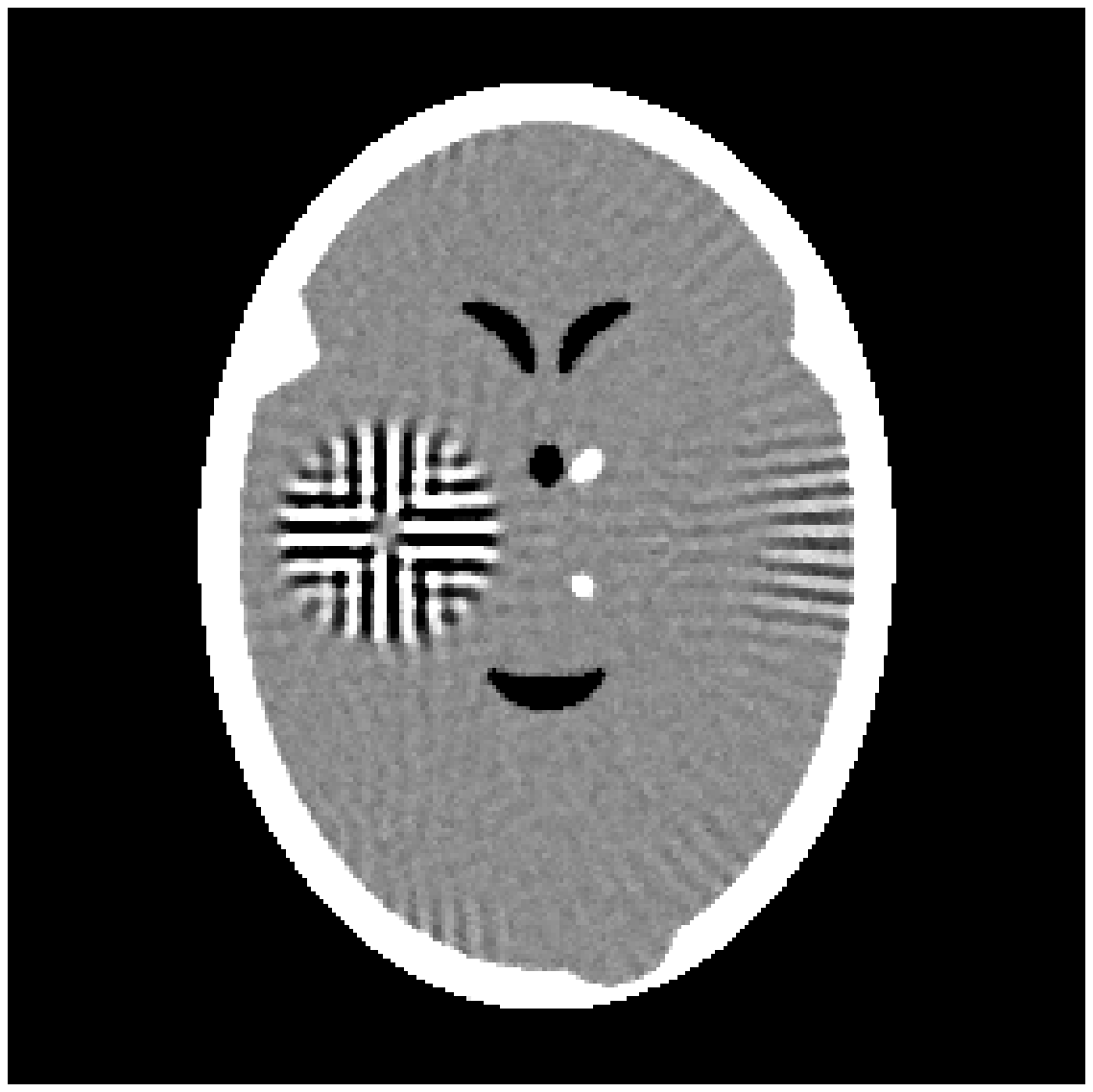}\\
TV-PPS&TV-PPS
\end{tabular}
\caption{The reconstruction results of ghost phantom without noise.
%The images in the first row are reconstructed by TV-S algorithm, while the images
% in the second row  by the proposed TV-PPS algorithm.
 From left to right, the images in each row are reconstructed from 112, 82 projections.}
\label{Fig7}
\end{figure}
\begin{table}[H]
\centering
\caption{The MSE, number of iteration and running time(RT) of images in Fig.~\ref{Fig7}}
\label{Tab4}
\begin{tabular}{p{2cm}<{\centering}|m{1cm}<{\centering}m{2cm}<{\centering}
m{1cm}<{\centering}m{2cm}<{\centering}m{1cm}<{\centering}m{1cm}<{\centering}m{1cm}
<{\centering}m{1cm}<{\centering}m{1cm}<{\centering}}
\hline
Algorithm   &TVS   &TV-PPS  &TVS    &TV-PPS\\%  &TVS    &MTVS \\
\hline
projections       & 112  &112    &  82  & 82 \\%  &60    &60\\
iterations  & 3    &  7    &13    &13  \\%  &27    &25\\
MSEs        &0.0102&0.0080&0.0164&0.0158\\%&0.0159&0.0119\\
RT(min)   &2.0529&5.1975&10.382&9.1644\\%&10.50 &10.1894\\
\hline
\end{tabular}
\end{table}

 By comparing the images in Fig.~\ref{Fig6}, \ref{Fig7} and numbers in Table~\ref{Tab3}, \ref{Tab4}, we can obtain the same conclusions that the
 proposed perturbation can not only improve qualities of reconstructed images, but also can
 accelerate the convergent speed. However, we can observe that
the reconstruction images suffer from artifacts regardless of the classic and the proposed algorithm when the projections is inadequate.

\section{Conclusion}
We investigated an optimization-based method to determine the
perturbation for superiorization algorithms.
We analyzed the convergence of the proposed superiorization algorithm.
Numerical experiments on different projection data for XCT image reconstruction
were conducted to validate the good performance of the proposed perturbation.
The experiments show that the perturbation determined by the proposed method
can improve the quality of reconstructed images and convergent speed of the superiorization algorithms.

%Superiorization algorithm is an emerging algorithm.
%There are some theoretical problems and applications  should be studied, such as the
%necessary conditions for the convergence of perturbed iteration, the initial aim $\phi(\hat{x})\leq \phi(x^\ast)$.
Since the superiorization methodology is a new algorithm for inverse problems,
the superiorized iteration needs to be studied further from the mathematical theory and applications\cite{Censor-Sup}.
In the future, we will investigate acceleration methods  and applications(parallel magnetic resonance imaging problem for instance)  of superiorization algorithm.

\bibliographystyle{ieeetr}
\bibliography{suanfa}
%\end{spacing}

%\end{CJK*}
\end{document}